\documentclass{lms}
\usepackage{amssymb}
\usepackage{amsmath}
\usepackage{mathabx}
\usepackage{hyperref}
\usepackage{graphicx}
\usepackage{mathrsfs}
\usepackage[curve]{xypic}
\usepackage{tikz}
\usetikzlibrary{arrows,decorations,shapes,shadows}
\usepackage{verbatim}
\usepackage[normalem]{ulem}
\newcommand{\grassman}{\mathbf G}

\let\cal\mathcal
\renewcommand{\setminus}{\smallsetminus}

\newcommand\C{{\mathbb C}}

\newcommand\N{{\mathbb N}}

\newtheorem{theorem}{Theorem}[section]
\newtheorem{proposition}[theorem]{Proposition}
\newtheorem{corollary}[theorem]{Corollary}
\newtheorem{lemma}[theorem]{Lemma}

\newnumbered{amalgamation}[theorem]{Amalgamation}
\newnumbered{example}[theorem]{Example}
\newnumbered{remark}[theorem]{Remark}
\newnumbered{definition}[theorem]{Definition}

\renewcommand{\int}{\mathop{\rm int}}

\newcommand{\norm}[1]{\lVert #1 \rVert}

 \usepackage{color}
  \usepackage{color}
 \usepackage{color}
 \usepackage{color}
 \usepackage{color}
 \usepackage{color}

\title[Minimal singularities are Lipschitz normally embedded]{Minimal surface singularities are  Lipschitz  normally embedded}
\author{Walter D Neumann, Helge M\o ller Pedersen, Anne Pichon}
\classno{14B05 (primary), 32S25, 32S05, 57M99 (secondary)}
\begin{document}
\maketitle

\begin{abstract} 
Any germ of a complex analytic space  is  equipped with two natural metrics: the {\it outer metric}  induced by the hermitian metric of the ambient space and the {\it inner metric}, which is the associated riemannian metric on the germ. We  show that  minimal surface singularities  are Lipschitz normally embedded (LNE),  i.e., the identity map is a bilipschitz homeomorphism between outer and inner metrics, and that they are the only rational surface singularities with  this property.  
\end{abstract}

\section{Introduction}

If $(X,0)$ is a germ of a complex  analytic space of pure dimension $ \dim (X,0)$, we denote by $m(X,0)$ its multiplicity and by $\mathop{\rm edim}(X,0)$ its embedding dimension. 

Minimal  singularities were introduced by J. Koll\'ar in \cite{K} as the germs of complex analytic spaces $(X,0)$ of pure dimension which are reduced,  Cohen-Macaulay,  whose tangent cone is reduced and whose  multiplicity is minimal in the sense that  Abhyankar's inequality 
$$m(X,0) \geq \mathop{\rm edim}(X,0) - \dim (X,0) + 1$$
is an equality (see \cite[Section 3.4]{K} or  \cite[Section 5] {BL}). 

In this paper, we  only deal with normal surfaces. In this case, minimality can be  defined as follows (\cite[Remark 3.4.10]{K}): a normal surface singularity $(X,0)$ is {\it minimal} if it is  rational   with a reduced minimal (also called fundamental) cycle. 

Minimal surface singularities play a key role in resolution theory of normal complex surfaces since they appear as central objects in the two main resolution algorithms: the resolution obtained as a finite sequence of normalized Nash modifications (\cite{S}), and the one obtained by a sequence of normalized blow-up of points (\cite{Z1939}, \cite{BL}). The question of the existence of a duality between these two algorithms, asserted by D. T. L\^e in \cite[Section 4.3]{L0} (see also \cite[Section 8]{BL}) remains  open, and the fact that minimal singularities seem to be the common denominator between them suggests the need of   a better understanding of  this class of surface germs.

In this paper, we study minimal surface singularities from the point of view of their Lipschitz geometries, and we show that they are characterized by a  remarkable metric property: they are Lipschitz normally embedded. Let us explain what this means.  

If $(X,0)$ is a germ of a complex variety, then any embedding $\phi\colon(X,0)\hookrightarrow
(\C^n,0)$ determines two metrics on $(X,0)$: the outer metric
$d_o(x_1,x_2):=\norm{\phi(x_1)-\phi(x_2)}$ 
    (i.e., distance in $\C^n$) and the inner metric
$d_i(x_1,x_2)$ defined as the length metric induced on $X$ by the  hermitian metric of $\C^n$. For all $x,y \in X, d_0(x,y) \leq d_i(x,y)$. 

\begin{definition} A germ of a complex   variety $(X,0)$ is {\it Lipschitz normally embedded} (LNE) if the identity map of $(X,0)$ is a bilipschitz homeomorphism between inner and outer metrics,  i.e.,  there exists a neigbourhood $U$ of $0$ in $X$ and a constant $K\geq 1$ such that for all $x,y \in U$
$$\frac{1}{K} d_i(x,y) \leq d_0(x,y).$$
\end{definition}

Let us now state our  main result:  
  
 \begin{theorem} \label{theorem:main}
 A rational surface singularity is  LNE if and
 only if is minimal. 
\end{theorem}
 
 The  proof is based on the characterization of LNE of normal surface singularities proved in \cite{NPP} {(see Theorem \ref{cor:complex characterization of normal embedding})}. We recall the statement in Section \ref{sec:characterization LNE}. 
 In  Section \ref{sec:minimal singularities}, we recall the definition of minimal singularities and we present the explicit description of the generic polar and
 discriminant curves of minimal surface singularities given  in
 \cite{S},  \cite{B1} and \cite{B2}  which will be used in the  next sections. In Section \ref{sec:inner rates}, we prove  results which will be used in   the proof of   the ``if" direction of Theorem \ref{theorem:main} in Section \ref{sec: main proof}.  The other
 direction of Theorem \ref{theorem:main} is proved in Section \ref{sec:rational + NE implies minimal}.
 
\vskip.1cm\noindent{\bf Acknowledgments.}  
Neumann was supported by NSF grant DMS-1608600.  Pedersen was supported
by FAPESP grant 2015/08026-4. Pichon was supported by the ANR project  LISA  17-CE40--0023-01  and by USP-Cofecub UcMa163-17. We are  very grateful for the hospitality and support of the following institutions: Columbia University, Institut de Math\'ematiques de Marseille,  FRUMAM Marseille, Aix Marseille Universit\'e and IAS Princeton.  We are very grateful to Jawad Snoussi and Bernard Teissier for helpful conversations, as well as the referee for helpful comments. 

\section{Characterization of LNE of a surface singularity}\label{sec:characterization LNE}

\subsection{Generic projections}
 
Let $\cal D$ be a $(n-2)$-plane in $\C^n$ and let $\ell_{\cal D}
\colon \C^n \to \C^2$ be the linear projection with
kernel $\cal D$. Suppose $(C,0)\subset (\C^n,0)$ is a complex curve germ.  There exists an open
dense subset $\Omega_C$  in the Grassmanian $ \grassman(n-2,\C^n)$  such that for $\cal D \in \Omega_C$, $\cal D$ contains no limit of secant lines to the curve $C$ (\cite{teissier}). 
 The  projection $\ell_{\cal D}$ is  said  to be \emph{generic for $C$} if  $\cal D \in \Omega_C$.

 Let now $(X,0)\subset (\C^n,0)$ be a normal surface singularity. We restrict ourselves to those $\cal D$ in the
Grassmanian $\grassman(n-2,\C^n)$ such that the restriction
$\ell_{\cal D}{\mid_{(X,0)}}\colon(X,0)\to(\C^2,0)$ is finite.
The \emph{polar curve}
$\Pi_{\cal D}$ of $(X,0)$ for the direction $\cal D$ is the closure in
$(X,0)$ of the singular locus of the restriction of $\ell_{\cal D} $
to $X \setminus \{0\}$. The \emph{discriminant curve} $\Delta_{\cal D}
\subset (\C^2,0)$ is the image $\ell_{\cal D}(\Pi_{\cal D})$ of the polar
curve $\Pi_{\cal D}$.

\begin{proposition}[{\cite[Lemme-cl\'e V 1.2.2]{teissier}}]\label{prop:generic} An open
dense subset $\Omega \subset \grassman(n-2,\C^n)$ exists 
such that: 
\begin{enumerate}
\item \label{cond:generic1} the family of  curve germs  $(\Pi_{\cal D})_{\cal D \in \Omega}$
is equisingular in terms of strong simultaneous resolution;
\item \label{cond:generic2} the discriminant curves  $ \Delta_{\cal D}=\ell_{\cal D}(\Pi_{\cal D})$, ${\cal D \in \Omega}$, form an equisingular family of reduced plane curves;
\item \label{cond:generic3}  for each $\cal D$, the projection $\ell_{\cal D}$ is generic for its polar curve $\Pi_{\cal D}$. 
\end{enumerate}
  \end{proposition}

\begin{definition}  \label{def:generic linear projection} 
The projection $\ell_{\cal D} \colon \C^n \to \C^2$
  is \emph{generic} for $(X,0)$ if $\cal D \in \Omega$.
\end{definition}

\subsection{Test curves} \label{subsec:rho}
  
Let  $\ell \colon (X,0) \to (\C^2,0)$ be a generic projection, let $\Pi $ be its polar curve and let $\Delta = \ell(\Pi)$ be its discriminant curve. Denote by $\rho'_{\ell} \colon Y_{\ell} \to \C^2$ the minimal composition of blow-ups of points starting with the blow-up of the origin which resolves the base points of the family of projections of generic polar curves  $(\ell(\Pi_{\cal D}))_{\cal D \in \Omega}$.

Let $E \subset Y$ be a complex curve in a complex surface $Y$ and let $E_1,\ldots,E_n$ be the irreducible components of $E$. We say {\it curvette} of $E_i$ for any smooth curve germ $(\beta,p)$ in $Y$, where $p$ is a point of $E_i$ which is a smooth point of $Y$ and $E$ {and} such that  $\beta$ and $E_i$ intersect transversely. 

If $G$ is a graph, we will denote by $V(G)$ its set of vertices and by $E(G)$ its set of edges.

\begin{definition}\label{def:test curve}
  We say \emph{$\Delta$-curve} for an exceptional curve in
    {$(\rho'_{\ell})^{-1}(0)$} intersecting the strict transform of $\Delta$.  
     Let us blow up all the intersection points between two
   $\Delta$-curves. We denote by $\sigma \colon Z_{\ell} \to Y_{\ell}$ and
   $\rho_{\ell}=\rho'_\ell\circ \sigma\colon Z_\ell\to\C^2$ the
   resulting morphisms (if no $\Delta$-curves intersect,
   $\rho_ {\ell} =  {\rho'_{\ell}}$).   The resolution graph $T$ of $\rho_{\ell}$ does not depend on $\ell$.  
  
      A  {\it $\Delta$-node} of $T$ is a vertex   of $T$ which represents a $\Delta$-curve. Let $T'$ be the subtree of $T$ defined as the union of all the simple paths in $T$ joining the root vertex to $\Delta$-nodes (so the complement $T \setminus T'$ consists of strings of valency $2$ vertices ended by a valency $1$ vertex). 

For $(i) \in V(T)$, let $C_i$  be the irreducible component of $\rho_{\ell}^{-1}(0)$ represented by $(i)$, so we have $\rho_{\ell}^{-1}(0) = \bigcup_{(i) \in V(T)} C_i$. 
Let $(i) \in V({T})$. We call \emph{test curve at $(i)$} (of $\ell$)  any complex curve germ $(\gamma,0) \subset (\C^2,0)$ such that 
\begin{enumerate}
\item $(i) \in V(T')$;
\item the strict transform $\gamma^*$   by  $\rho_\ell$ is a curvette of $C_i$  intersecting $C_i$ at a smooth point of $\rho_{\ell}^{-1}(0)$;
\item $\gamma^* \cap \Delta^* = \emptyset$. 
\end{enumerate}
\end{definition}

\subsection{Nash modification and lifted Gauss map}
\begin{definition} \label{def:lifted Gauss map}
Let $\lambda \colon X\setminus\{0\} \to \grassman(2,\C^n)$ be the Gauss  map,  
which sends $x \in X\setminus\{0\}$ to the tangent plane $T_xX$. The
closure $N X$ of the graph of $\lambda$ in
$X \times \grassman(2,\C^n)$ is a reduced analytic surface. By
definition, the \emph{Nash modification} of $(X,0)$ is the induced
morphism $\mathscr N \colon NX \to X$. The \emph{lifted Gauss map} is the morphism
$\widetilde{\lambda} \colon NX \to \grassman(2,\C^n) $ defined as the
restriction to $NX$ of the projection of $X \times \grassman(2,\C^n)$
on the second factor. 
\end{definition}
 
\begin{lemma}[{\cite[Part III, Theorem 1.2]{S}}, {\cite[Section
    2]{GS}}]\label{le:nash}
  A morphism $f \colon Y \to X$ factors through Nash modification if
  and only if it has no base points for the family of polar curves.
\end{lemma}

\subsection{Principal components}
\begin{definition} \label{def:node G resolution}  Let $\pi_0 \colon
  X_0 \to X$ be the minimal good  resolution of $X$ which factors
  through both the Nash modification and the blow-up of the maximal
  ideal and let $G_0$ be its resolution graph. For each vertex
  {$(v)\in V(G_0)$}  we denote by $E_v$ the corresponding  irreducible
  component of $\pi_0^{-1}(0)$. A vertex {$(v)\in V(G_0)$} such that $E_v$ is an irreducible component of the blow-up of the maximal ideal (resp.\ an exceptional curve of the Nash modification)  is called an {\it $\cal L$-node} (resp.\ a {\it $\cal P$-node}) of $G_0$. 
\end{definition}

\begin{definition}\label{def:G'} Consider the subgraph $G'_0$ of $G_0$ defined as the union of all simple paths in  $G_0$ connecting  pairs of  vertices among $\cal L$- and $\cal P$-nodes. Let $\ell \colon (X,0) \to (\C^2,0)$ be a generic projection and let  $\gamma$ be a test curve for $\ell$. A component $\widehat{\gamma}$ of $\ell^{-1}(\gamma)$ is called {\it principal} if its  strict transform by $\pi_0$ {either is a curvette of a component $E_v$ with $(v) \in V(G'_0)$ or intersects $\pi_0^{-1}(0)$ at an intersection between two exceptional curves  $E_v$ and $E_{v'}$ such that both  $(v)$ and $(v')$ are  in $V(G'_0)$.   }
\end{definition}

\subsection{Inner and outer contact exponents}  We will  use the ``big-Theta" asymptotic notation of Bachman-Landau:  given two function germs $f,g\colon ([0,\infty),0)\to ([0,\infty),0)$ we say $f$ is \emph{big-Theta} of $g$ and we write   $f(t) = \Theta (g(t))$ if there exist real numbers $\eta>0$ and $K \ge 1$ such that for all $t$ with $f(t)\le\eta$,  $\frac{1}{K }g(t) \leq f(t) \leq K g(t)$.

  Let $\mathbb S^{2n-1}_{\epsilon} = \{ x \in \C^n \colon  \norm{x}_{\C^n} = \epsilon\}$.  
  
  \begin{definition}
  Let $(\gamma_1,0)$ and $(\gamma_2,0)$ be two germs of complex curves in $(\C^{n},0)$.  The {\it outer  contact exponent} between $\gamma_1$ and $\gamma_2$ is the rational number $q_{out}=q _{out}(\gamma_1, \gamma_2) \geq 1$ defined by: $ d_o(\gamma_1 \cap \mathbb S^{2n-1}_{\epsilon}, \gamma_2 \cap \mathbb S^{2n-1}_{\epsilon}) =  \Theta(\epsilon^{q_{out}})$. 
  \end{definition}
  
    \begin{definition} Let $(X,0)$ be a complex surface germ and let $(\gamma_1,0)$ and $(\gamma_2,0)$ be two germs of complex curves  
  inside $(X,0)$.  The {\it inner   contact exponent} between $\gamma_1$ and $\gamma_2$ on $(X,0)$ is the rational number $q_{inn}=q _{inn}(\gamma_1, \gamma_2) \geq 1$ defined by: $ d_i(\gamma_1 \cap \mathbb S^{2n-1}_{\epsilon}, \gamma_2 \cap \mathbb S^{2n-1}_{\epsilon}) =  \Theta(\epsilon^{q_{inn}})$, where $d_{i}$ means inner distance in $(X,0)$ as before.  
\end{definition}

We are now ready to state {the characterization theorem for LNE for 
 a complex normal surface germ. {The version we give here is a sightly weaker version than  \cite[Theorem 3.8]{NPP}, in which we use a restricted version of test curves (the so-called nodal test curves). We just need this weaker version to prove LNE for minimal  surface singularities.}

\begin{theorem}
  \label{cor:complex characterization of normal embedding} A normal surface germ  $(X,0)$  is LNE if and only if the following conditions are satisfied for {all  generic projections} $\ell \colon (X,0) \to (\C^2,0)$ and  test curves $(\gamma,0) \subset (\C^2,0)$: 
\begin{itemize}
\item[($1^*$)] \label{iso}  {for each  principal component $\widehat{\gamma}$ of $\ell^{-1}(\gamma)$, $mult(\widehat{\gamma})=mult({\gamma})$ where $mult$ means multiplicity at $0$; }
\item[($2^*$)] \label{vertical} for {all pairs} $(\widehat{\gamma}_1,\widehat{\gamma}_2)$ of distinct
principal components of $\ell^{-1}(\gamma)$,
  $q_{inn}(\widehat{\gamma}_1,\widehat{\gamma}_2) = q_{out}(\widehat{\gamma}_1,\widehat{\gamma}_2)$.
\end{itemize}
 \end{theorem}}

\section{Minimal Singularities} \label{sec:minimal singularities}

In this section, we give the definition of minimal singularities and specify the case of surfaces. We then present the description by Spivakovsky  (\cite{S}) of the minimal resolution which  factors through Nash modification  and  the description by Bondil (\cite{B1,B2}) of the {morphism} $\rho'_{\ell} \colon Y_{\ell} \to \C^2$ introduced in Section \ref{subsec:rho}, where $\ell \colon (X,0)\to (\C^2,0)$ is a fixed generic plane projection.  

 Let us first recall  the definition  of  the minimal  cycle $Z_{min}$  (also called Artin fundamental cycle) of a normal surface singularity $(X,0)$. Let $\pi \colon{(\widetilde X,E)}\to{(X,0)}$ be  the minimal resolution  of $X$ and let $E_1,\ldots,E_r$ be the irreducible components of the exceptional divisor $E=\pi^{-1}(0)$.    The {\it minimal cycle} $Z_{min}$ is the minimal element of the set of divisors $Z=\sum_{i=1}^r m_i E_i$    whose coefficients $m_i$ are strictly positive integers and such that $\forall j=1,\ldots,r,$ $Z\cdot E_j \leq 0$. A {\it reduced} minimal cycle means that $Z_{min}=\sum_i^r E_i$, i.e.,  $m_i = 1$ for all  $i=1,\ldots,r$.  
 
 If $f\colon{(X,0)}\to{(\C,0)}$ is  an analytic function, then its total transform $(f)=(f  \circ \pi)^{-1}(0)$ decomposes into $(f) = Z(f) + f^*$ where $f^*$ is the strict transform and  $Z(f)$ a positive divisor with support on $E$. {If $\pi$  is a good resolution of $(X,0)$ (i.e., $\pi^{-1}(0)$ consists of smooth transversal irreducible components intersecting transversely at double points)}, then for each $j=1,\ldots,r$, one has  $(f)\cdot  E_j = 0$ {(\cite{La2})}. Hence  $Z(f)\cdot E_j \leq 0$ for all $j=1,\ldots,r$. If $h \colon (X,0) \to (\C,0)$ is a generic linear form, then $Z(h)$ is the minimal element among divisors $Z(f)$, and $Z_{min}\leq Z(h)$.  
 
 Recall that if $(X,0)$ is a  rational surface singularity, then its minimal resolution $\pi$ is a good resolution,  $\pi$ resolves the basepoints of the family of generic linear forms (or equivalently, it factors through the blow-up of the maximal ideal),  and $Z(h)=Z_{min}$ (\cite{A}). 
 
 We now recall the characterization of minimal surface singularities proved by Koll\`ar:
 
 \begin{proposition}[{\cite[Remark 3.4.10]{K}}] \label{prop:minimal iff rational and reduced fundamental cycle}
 A normal surface singularity is minimal if and only if it is rational with reduced fundamental cycle. 
 \end{proposition}


In \cite{S}, Spivakovsky gives the following  combinatorial characterization of the dual resolution graph of minimal singularities. Let  $(X,0)$ be a normal surface singularity, let ${\pi'}\colon X' \to{X}$ be the minimal good resolution of $(X,0)$ and let  $G$ be its dual graph.   If $(v) \in V(G)$, we  denote by $E_v$ the corresponding irreducible component of the exceptional divisor $(\pi')^{-1}(0)$, we set  $w(v)=E_v^2$  and we denote by $\nu(v)$ the valence of $(v)$, i.e., the number of edges {in $G$} adjacent to $(v)$.  

\begin{proposition}[{\cite[{Part II, Remark 2.3}]{S}}]  \label{prop:minimal} A surface singularity is minimal if and only if $G$ is a tree of rational curves and for all vertices $(v)\in V(G)$, $-w(v)\geq \nu(v)$ (in which case, $\pi'$ coincides with $\pi$.)
\end{proposition}

\begin{remark}\label{rk:L-nodes}
A consequence of Proposition \ref{prop:minimal} is that if  $(X,0)$ is minimal, then the $\cal L$-nodes (Definition \ref{def:node G resolution}) in $G$ are the vertices $(v)$ such that  $-w(v)> \nu(v)$. In particular, it implies that every leaf of $G$ is an $\cal L$-node.  
 \end{remark}

Spivakovsky introduced the function $s \colon{V(G)}\to{\N}$ defined as follows: $s(v)$ is 
 the number of vertices on the shortest path  {in $G$} from $(v)$   to an $\cal L$-node. So $s(v)=1$ if and only if $(v)$ is an $\cal L$-node. Since minimal singularities
are rational they can be resolved by only blowing up points, as
Tjurina showed in \cite{Tyu}, and $s(v)$ is the number of blow-ups it
takes before $E_v$ appears in the successive exceptional
divisors.

 We now state a result of Spivakovsky in a formulation inspired by Bondil in \cite{B1} which will enable one to describe $\pi_0$ and $G_0$ from the graph $G$. 
 
\begin{theorem}[{\cite[Part III, Theorem 5.4]{S}}] \label{thm:Pi} \label{thm:spi} Let $(X,0)$  be a minimal surface singularity. Let  $\ell \colon (X,0) \to (\C^2,0)$ be a generic projection and let $\Pi$ be its polar curve. Let ${\pi} \colon \widetilde{X}\to X $ be the minimal resolution of $(X,0)$. Consider the  cycle $S:=\sum s(v)E_v$, where the $E_v$ are the irreducible components of $\pi^{-1}(0)$.   Then the strict transform $\Pi^*$ of   $\Pi$ by ${\pi}$ is smooth. It  consists of  exactly   $-(S+E_v)\cdot E_v-2$ curvettes of each $E_v$,  one component of which goes through each intersection  point $E_v \cap E_w$ for which $s(v)=s(w)$. Moreover, the latter intersection points are the only basepoints of the family of generic polars $(\Pi_{\cal D})_{\cal D \in \Omega}$ and each of them is resolved by  one blow-up. 
\end{theorem}

\begin{definition}
Following the terminology of \cite{S}, an edge of
    $G$ joining two vertices $(v)$ and $(w)$ is \emph{central} if
    $s(v)=s(w)$, and a vertex $(v)$ is  \emph{central} if there are at
    least two neighbouring vertices $(w),(w')$ such that
    $s(v)-1=s(w)=s(w')$.
\end{definition}

\begin{remark}   Let $(v)$ be a non-central vertex which is not
    an $\cal L$-node and not adjacent to a central edge. {Set $r = \nu(v)-1$}.  Then
      $(v)$ has one neighbour vertex $(v_0)$ with 
    $s(v_0)=s(v)-1$, and $r$   neighbour vertices $(v_1),\dots,(v_r)$
    satisfying $s(v_i)=s(v)+1$ for all $i$. Let $h \colon (X,0) \to
    (\C,0)$ be a  generic linear form and let $(h)$ be its total
    transform by  $\pi$. Since the minimal cycle of $(X,0)$ is
    reduced, the equality $(h)\cdot E_v=0$ gives $E_v^2+r+1=0$, i.e.,
    $E_v^2=-(r+1)$. We then have $-(S+E_v)\cdot E_v-2 =
    -\sum_{i=1}^r s(v_i)E_{v_i}\cdot E_v - s(v)E_v^2-s(v_0)E_{v_0}\cdot E_v -E_v^2-2 = 0$.

Using this, Theorem \ref{thm:Pi} says that for each central edge  there is one component of $\Pi^*$ through the intersection point of the corresponding curves and that for each central vertex $(v)$, there is at least one component of $\Pi^*$ which is a curvette of $E_v$.  {Any other component of $\Pi^*$ goes through $\cal L$-curves. }
\end{remark}

As a consequence of Lemma \ref{le:nash} and Theorem \ref{thm:Pi}, we  obtain the following  explicit description of $\pi_0$ and of $G_0$. (In \cite{B1}, Bondil shows that $\pi_0$ is actually the minimal resolution of $(X,0)$ obtained by only blowing up points.)
 
 \begin{corollary} \label{cor:minimal resolution} The  minimal good resolution $\pi_0 \colon X_0 \to X$ of $X$ which factors through Nash modification and the blow-up of the maximal ideal is obtained by composing $\pi$ with the blow-up of  each intersection  point $E_v \cap E_w$ corresponding to a central edge.   
\end{corollary}

\begin{remark}\label{rk:G_0=G'_0}{Recall (Definition \ref{def:G'}) that  $G'_0$ is the subgraph of $G_0$ defined as the union of all simple paths in  $G_0$ connecting  pairs of  vertices among $\cal L$- and $\cal P$-nodes. A consequence of Corollary  \ref{cor:minimal resolution} and of Remark \ref{rk:L-nodes} is that for a minimal singularity, we have $G'_0=G_0$.} \end{remark}

 We now present a more precise description of the polar curve and of the discriminant curve given by Bondil in \cite{B1} which will lead to the explicit description of the resolution tree $T_0$ of $\rho'_{\ell}\colon Y_{\ell} \to \C^2$ and its relation with $G_0$ in Corollary \ref{cor:graph map}.  

An  $A_n$-\emph{curve}  is a germ of an analytic curve isomorphic to the plane curve $y^2+x^{n+1}=0$. If $n$ is  odd, then $A_n$ consists of  a pair of smooth curves with  contact exponent $\frac{n+1}{2}$ while  if $n$ is  even, $A_n$ is an irreducible curve.

\begin{theorem}[{\cite{B1,B2}}] \label{genericpolar}  
Let $(X,0)$ be a minimal singularity and let $\Pi$ be the polar of a generic linear projection. Then  
\begin{enumerate}
\item $\Pi$ decomposes as a union of $A_{n_i}$-curves $\Pi=\bigcup_i \Pi_i$  and each $\Pi_i$ meet a single irreducible component $E_{v_i}$ of the exceptional divisor of $\pi_0$
\item If   {$E_{v_i}$  comes from blowing up a central edge $(v'_i)-(v''_i)$, then $\Pi_i$ is an (irreducible) $A_{2s(v'_i)}$-curve.}   Otherwise $\Pi_i$ consists of two smooth curves forming an $A_{2s(v_i)-1}$-curve. 
 \item \label{item:contact0} The contact exponent between $\Pi_i$ and $\Pi_j$ equals the minimal value   of $s(v)$  on the shortest path in $G_0$ between the vertices ${v_i}$ and ${v_j}$. 
\end{enumerate}
  \end{theorem}

\begin{example}\label{example0} 
{
Let $(X,0)$ be a minimal singularity  with the following resolution graph:

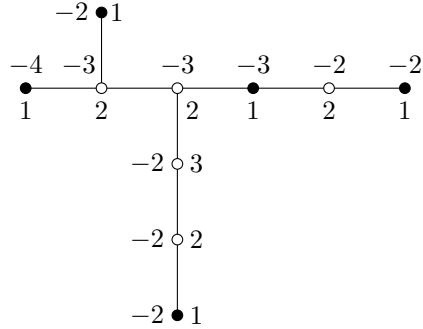
\begin{center}
\begin{tikzpicture}
 
   \draw[thin ](-1,0)--(4,0);
      \draw[thin ](1,0)--(1,-3);
       
                \draw[thin ](0,1)--(0,0);
           \draw[fill=black   ] (0,1)circle(2pt);

           \draw[fill=black   ] (-1,0)circle(2pt);
           \draw[fill=white  ] (0,0)circle(2pt);
          \draw[fill=white ] (1,0)circle(2pt); 
           \draw[fill=black] (2,0)circle(2pt);
  \draw[fill=white ] (3,0)circle(2pt);
        \draw[fill =black ] (4,0)circle(2pt);
        
          \draw[fill=white ] (1,-1)circle(2pt);
         \draw[fill=white ] (1,-2)circle(2pt);
        \draw[fill =black ] (1,-3)circle(2pt);  
 

\node(a)at(-1,0.3){   $-4$};
\node(a)at(-0.3,0.3){   $-3$};
\node(a)at(-0.4,1){   $-2$};
\node(a)at(0.2,1){   $1$};
\node(a)at(1,0.3){   $-3$};
\node(a)at(2,0.3){   $-3$};
\node(a)at(3,0.3){   $-2$};
\node(a)at(4,0.3){   $-2$};
\node(a)at(0.6,-1){   $-2$};
\node(a)at(0.6,-2){   $-2$};
\node(a)at(0.6,-3){   $-2$};
 
\node(a)at(-1,-0.3){   $1$};
\node(a)at(0,-0.3){   $2$};
\node(a)at(1.2,-0.3){   $2$};
\node(a)at(2,-0.3){   $1$};
\node(a)at(3,-0.3){   $2$};
\node(a)at(4,-0.3){   $1$};
\node(a)at(1.25,-1){   $3$};
\node(a)at(1.25,-2){   $2$};
\node(a)at(1.25,-3){   $1$};
 

  \end{tikzpicture} 
  \end{center}
 
 The negative  weights are the self-intersections of the exceptional
 curves and the positive weights are the values of $s$. The $\cal L$-nodes are the black vertices.

The  graph on the picture below is the resolution graph $G_0$ determined by Theorem \ref{genericpolar}. The graph is decorated with arrows representing  the components of the strict transform of the polar curve $\Pi$ of a generic plane projection. The weights at the vertices which are in $G$  are the values of the function $s$. 
The gray node represents an exceptional curve obtained by
blowing up  the intersection point of two exceptional curves corresponding to a central edge. The polar curve $\Pi$  consists of  five pairs $\Pi_1,\ldots,\Pi_5$ of smooth components and one component $\Pi_6$
with multiplicity $2$.   The curves $\Pi_1$ and  $\Pi_2$ are $A_1$-curves, $\Pi_3$ and $\Pi_5$ are $A_3$-curves,  $\Pi_4$  is an $A_5$-curve and
$\Pi_6$ is an $A_4$-curve.  By Theorem \ref{genericpolar}, the contact between the two branches of $\Pi$ is  the minimal value of $s$ along
the shortest path between them, e.g.\ the contact between $\Pi_4$ and
$\Pi_5$ is $2$ and the contact between $\Pi_3$ and $\Pi_5$ is $1$.

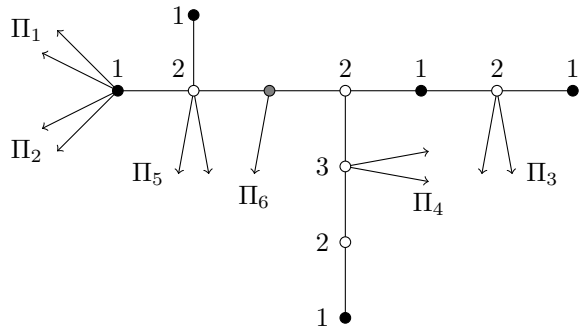
\begin{center}
\begin{tikzpicture}
 
   \draw[thin ](-2,0)--(4,0);
      \draw[thin ](1,0)--(1,-3);
      \draw[thin,>-stealth,->](-2,0)--+(-0.8,0.8);
       \draw[thin,>-stealth,->](-2,0)--+(-0.8,-0.8);       
        \draw[thin,>-stealth,->](-2,0)--+(-1,-0.5);
          \draw[thin,>-stealth,->](-2,0)--+(-1,0.5);
          
          \node(a)at(-3.2,0.8){   $\Pi_1$};
       \node(a)at(-3.2,-0.8){   $\Pi_2$};

             \draw[thin,>-stealth,->](0,0)--+(-0.2,-1.1);
               \node(a)at(-0.2,-1.4){   $\Pi_6$};
               
                 \draw[thin,>-stealth,->](-1,0)--+(-0.2,-1.1);
          \draw[thin,>-stealth,->](-1,0)--+(0.2,-1.1);
               \node(a)at(-1.6,-1.1){   $\Pi_5$};

              \draw[thin,>-stealth,->](1,-1)--+(1.1,-0.2);
          \draw[thin,>-stealth,->](1,-1)--+(1.1,0.2);
               \node(a)at(2.1,-1.5){   $\Pi_4$};
               
            \draw[thin,>-stealth,->](3,0)--+(-0.2,-1.1);
          \draw[thin,>-stealth,->](3,0)--+(0.2,-1.1);
               \node(a)at(3.6,-1.1){   $\Pi_3$};

      \draw[thin ](-1,1)--(-1,0);
           \draw[fill=black   ] (-1,1)circle(2pt);

               \draw[fill=black ] (-2,0)circle(2pt);
           \draw[fill=white   ] (-1,0)circle(2pt);
           \draw[fill=gray   ] (0,0)circle(2pt);
          \draw[fill=white ] (1,0)circle(2pt); 
           \draw[fill=black] (2,0)circle(2pt);
  \draw[fill=white ] (3,0)circle(2pt);
        \draw[fill =black ] (4,0)circle(2pt);
        
          \draw[fill=white ] (1,-1)circle(2pt);
         \draw[fill=white ] (1,-2)circle(2pt);
        \draw[fill =black ] (1,-3)circle(2pt);  
 

\node(a)at(-2,0.3){   $\it 1$};

\node(a)at(-1.2,0.3){   $\it 2$};

\node(a)at(-1.2,1){   $\it 1$};


\node(a)at(1,0.3){   $\it 2$};

\node(a)at(2,0.3){ $\it 1$};

\node(a)at(3,0.3){   $\it 2$};
\node(a)at(4,0.3){   $\it 1$};

\node(a)at(0.7,-3){ $\it 1$}; 
 \node(a)at(0.7,-2){ $\it 2$}; 

\node(a)at(0.7,-1){ $\it 3$}; 
           
  \end{tikzpicture} 
  \end{center}
}
\end{example}

Using the fact  that each  branch of $\Pi$ is isomorphic to a plane curve and that the restriction $\ell |_{\Pi} \colon \Pi \to \Delta$ is generic, Bondil deduces from Theorem \ref{genericpolar} the following description of the discriminant curve:

\begin{theorem}[{\cite{B1,B2}}]\label{prop:discriminant} 
\begin{enumerate}
\item \label{item:contact} The discriminant curve $\Delta$ of a generic projection $\ell$ of $(X,0)$ is a  union of $A_{n}$-curves  in one-to-one correspondence with the curves $\Pi_i$ of Theorem \ref{genericpolar}, and  {the contact between any two of them equals}    that of the corresponding  $C_i$'s;
\item  the minimal resolution of $\Delta$ is the resolution ${\rho'_{\ell}} \colon Y_{\ell} \to \C^2$   which resolves the base points of the family of projected generic polar curves $(\ell(\Pi_{\cal D}))_{\cal D \in \Omega}$.
\item \label{iii} The minimal resolution $\pi \colon \widetilde{X} \to X$ of $(X,0)$  is a composition of blow-ups of points $\pi = \pi_1 \circ \ldots \circ \pi_n$ and it resolves the polar curve $\Pi$ of any generic projection $\ell  \colon (X,0) \to (\C^2,0)$. Moreover, the resolution of the discriminant curve $\Delta = \ell(\Pi)$ is a union of  blow-ups  $\rho =\rho_1 \circ \ldots \circ \rho_n$ starting with the blow-up $\rho_1$ of the origin of $\C^2$ such that we have a commutative diagram consisting of successive fiber products:

$$\xymatrix{\widetilde{X} =X_n \ar@{->}[r]^{\pi_n} \ar@{->}[d]^{\ell_n} & X_{n-1} \ar@{->}[r]^{\pi_{n-1}}  \ar@{->}[d]^{\ell_{n-1}} & \cdots  \ar@{->}[r]^{\pi_2} & X_1 \ar@{->}[r]^{\pi_{1}}  \ar@{->}[d]^{\ell_1} &X  \ar@{->}[d]^{\ell}   \\
Y_n  \ar@{->}[r]^{\rho_n}  & Y_{n-1} \ar@{->}[r]^{\rho_{n-1}}    & \cdots   \ar@{->}[r]^{\rho_2} & Y_1 \ar@{->}[r]^{\rho_{1}}   &\C^2
} 
$$
\end{enumerate}
\end{theorem}

{
Consider  the morphism $\sigma' \colon Y' \to Y_n$ such that $\rho \circ \sigma'$ is the minimal good resolution of $\Delta$.   By (i) of Theorem \ref{prop:discriminant}, $\sigma'$ consists of blowing up  each intersection point $q=\Delta^*_0 \cap \rho^{-1}(0)$, where $\Delta_0$ is a component of $\Delta$ which is a  $A_n$-curve with $n$  even and then the intersection point $q'=\Delta_0^* \cap E_q$ where $E_q$ is the exceptional $\Bbb P^1$-curve created by blowing up $q$.
By (ii) of  Theorem \ref{prop:discriminant}, we have  $Y'=Y_{\ell}$ and $\rho'_{\ell} = \rho \circ \sigma'$. Moreover,  the  dual  graph $T_0$ of the resolution $\rho'_{\ell} $ is determined  from the resolution graph $G$ of $\pi$ using (i) of Theorem \ref{prop:discriminant}.

 Let $T'_0$ be the subgraph of $T_0$  consisting of the union of paths joining the root vertex to $\Delta$-curves, so  $T_0 \setminus T'_0$ consists of  isolated vertices corresponding to the curves $E_q$. Let $\alpha \colon X_0 \to \widetilde{X}$ be the morphism defined by  $\pi_0  = {\pi \circ \alpha}$. 
 Consider the morphism $\ell_n \colon \widetilde{X} \to Y_n$ introduced in (iii) of Theorem \ref{prop:discriminant}. The restriction $\ell_n \mid_{\pi^{-1}(0)} \colon \pi^{-1}(0) \to \rho^{-1}(0)$ lifts to a unique morphism $\ell' \colon \pi_0^{-1}(0) \to \bigcup_{u \in V(T'_0)} C_u$  such that $\ell_n \circ \alpha = \sigma' \circ \ell'$.  The image by $\ell'$ of each component of $\pi_0^{-1}(0)$ is a curve, so we have an  induced graph-map $L \colon G_0 \to T'_0$, i.e.,  $L(V(G_0)) =V(T'_0)$ and the image by $L$ of an edge $(v,v')$
   of $G_0$ is the edge $(L(v),L(v'))$. 
 
 Let us extend the function $s \colon G \to \N^*$ to a function $s \colon G_0 \to \frac{1}{2}  \N^*$ by setting $s(v)=s(v_1)+1/2$ for each vertex $(v)$ obtained by blowing up a central edge $(v_1)-(v_2)$ of $G$. 
 
 Let us define a function $\widehat{s}   \colon T'_0 \to \frac{1}{2}  \N^*$ as follows. If $(v)$ is a vertex representing a curve of $\rho^{-1}(0)$, $\widehat{s}(v)$ is the number of vertices on the shortest path from $(v)$ to the root vertex. {Otherwise}, $\widehat{s}(v) = s(w)+1/2$ where  $(w)$ is the vertex of $T'_0$ adjacent to $(v)$.

\begin{corollary} \label{cor:graph map}   
 \begin{enumerate}
 \item \label{cor:(i)} For each vertex $(v)$ of $G_0$,  $s(v) = \widehat{s}(L(v))$;  
\item \label{cor:(ii)}  For $n \in \frac{1}{2}  \N^*$, denote by $T'_0(\widehat{s} > n)$ the maximal subtree of $T'_0$ such that all vertices of $T'_0(\widehat{s}  > n)$ have $\widehat{s} > n$ and let $G_0(s > n)$ be the maximal subgraph of $G_0$ such that all vertices of $G_0(s  > n)$ have {$s > n$}. Then for every connected component  $\tau$ of $T'_0(\widehat{s} > n)$, $L^{-1}(\tau)$ is a connected component of  $G_0(s > n)$. 
\end{enumerate}
\end{corollary}

\begin{proof}
{ This is a direct consequence of Theorems \ref{genericpolar} and   \ref{prop:discriminant}. In particular, Point \ref{cor:(ii)} is a consequence of \ref{item:contact0} of Theorem \ref{genericpolar} and of  \ref{item:contact} of Theorem \ref{prop:discriminant}.}
 \end{proof}

  \begin{example} We continue with the minimal singularity introduced in Example \ref{example0}. 
  The right tree in the picture below is the  tree $T_0$. It is obtained by using Theorem \ref{prop:discriminant}. The arrows represent the components of
  $\Delta = \bigcup_{i=1}^6 \Delta_i$ where  $\Delta_i=\ell(\Pi_i)$,
  $i=1,\ldots,6$.  Each vertex of $T'_0$ is weighted by the value of $\widehat{s}$.  The graph on the left is the graph $G_0$ determined in Example   \ref{example0}. Each vertex of $G_0$ is weighted by the value of the extended function $s\colon G_0 \to \frac{1}{2}  \N^*$, and the graph-map $L$ from $G_0$ to $T'_0$ described in Corollary \ref{cor:graph map} sends vertices and edges horizontally. 
  

  
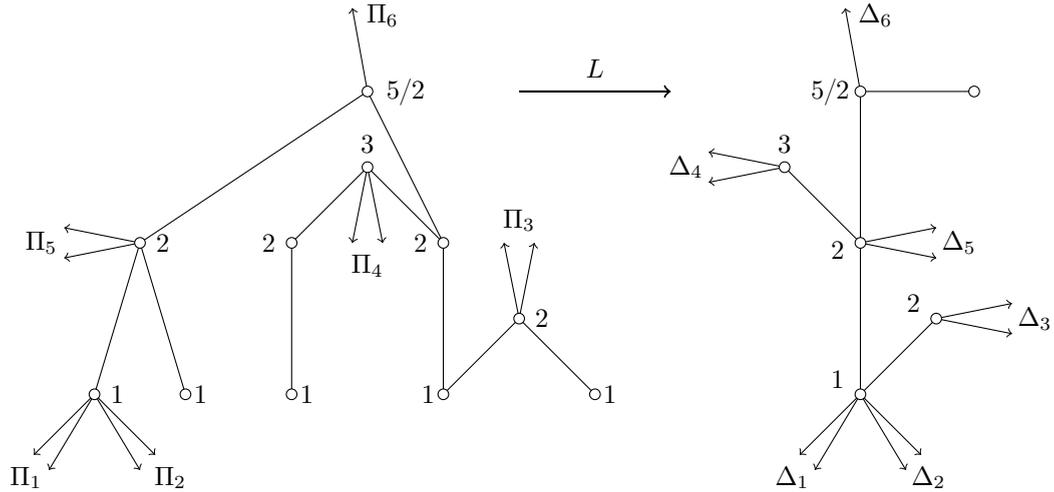
\begin{center}
\begin{tikzpicture}
 \begin{scope}[xshift=-1.5cm]
   \draw[thin ](2,0)--(3,-2);
     \draw[thin ](2,0)--(-1,-2);
       \draw[thin ](2,-1)--(1,-2);
       \draw[thin ](2,-1)--(3,-2);
          \draw[thin ](3,-2)--(3,-4);
           \draw[thin ](4,-3)--(3,-4);
             \draw[thin ](4,-3)--(5,-4);
               \draw[thin ](1,-2)--(1,-4);
                 \draw[thin ](-1,-2)--(-0.4,-4);
                 \draw[thin ](-1,-2)--(-1.6,-4);
                 
\draw[thick, >-stealth,-> ](4,0)--(6,0);
 \node(a)at(5,0.3){   $L$};
                 
  \draw[thin,>-stealth,->](-1.6,-4)--+(+0.8,-0.8);
       \draw[thin,>-stealth,->](-1.6,-4)--+(-0.8,-0.8);
        \draw[thin,>-stealth,->](-1.6,-4)--+(-0.6,-1);
          \draw[thin,>-stealth,->](-1.6,-4)--+(+0.6,-1);
            \node(a)at(-2.5,-5.1){   $\Pi_1$};
       \node(a)at(-0.6,-5.1){   $\Pi_2$};
       
         \draw[thin,>-stealth,->](2,-1)--+(+0.2,-1);
       \draw[thin,>-stealth,->](2,-1)--+(-0.2,-1);
        \node(a)at(2,-2.3){   $\Pi_4$};
        
          \draw[thin,>-stealth,->](4,-3)--+(+0.2,1);
       \draw[thin,>-stealth,->](4,-3)--+(-0.2,1);
        \node(a)at(4,-1.7){   $\Pi_3$};
        
         \draw[thin,>-stealth,->](2,0)--+(-0.2,1.1);
 \node(a)at(2.2,1){   $\Pi_6$};
 
  \draw[thin,>-stealth,->](-1,-2)--+(-1,0.2);
             \draw[thin,>-stealth,->](-1,-2)--+(-1,-0.2);
               \node(a)at(-2.3,-2){   $\Pi_5$};

        \draw[fill=white] (2,0)circle(2pt);
      \draw[fill=white] (-1,-2)circle(2pt);
        \draw[fill=white] (1,-2)circle(2pt);
          \draw[fill=white] (3,-2)circle(2pt);
            \draw[fill=white] (4,-3)circle(2pt);
             \draw[fill=white] (3,-4)circle(2pt);
              \draw[fill=white] (5,-4)circle(2pt);
                   \draw[fill=white] (2,-1)circle(2pt);
                    \draw[fill=white] (1,-4)circle(2pt);
                      \draw[fill=white] (-0.4,-4)circle(2pt);
                       \draw[fill=white] (-1.6,-4)circle(2pt);

\node(a)at(2.5,0){ $\it 5/2$};
\node(a)at(2,-0.7){ $\it 3$};
\node(a)at(2.7,-2){ $\it 2$};
\node(a)at(0.7,-2){ $\it 2$};
\node(a)at(-0.7,-2){ $\it 2$};
\node(a)at(4.3,-3){ $\it 2$};

\node(a)at(5.2,-4){ $\it 1$};
\node(a)at(2.8,-4){ $\it 1$};
\node(a)at(1.2,-4){ $\it 1$};
\node(a)at(-0.2,-4){ $\it 1$};
\node(a)at(-1.3,-4){ $\it 1$};

\end{scope}


  \draw[thin ](7,0)--(7,-4);
   \draw[thin ](7,0)--(8.5,0);
    \draw[thin ](7,-2)--(6,-1);
     \draw[thin ](7,-4)--(8,-3);

      \draw[thin,>-stealth,->](7,-4)--+(+0.8,-0.8);
       \draw[thin,>-stealth,->](7,-4)--+(-0.8,-0.8);
        \draw[thin,>-stealth,->](7,-4)--+(-0.6,-1);
          \draw[thin,>-stealth,->](7,-4)--+(+0.6,-1);
            \node(a)at(6.1,-5.1){   $\Delta_1$};
       \node(a)at(7.9,-5.1){   $\Delta_2$};

             \draw[thin,>-stealth,->](7,-2)--+(+1,0.2);
             \draw[thin,>-stealth,->](7,-2)--+(+1,-0.2);
               \node(a)at(8.3,-2){   $\Delta_5$};
               
         \draw[thin,>-stealth,->](8,-3)--+(+1,0.2);
                  \draw[thin,>-stealth,->](8,-3)--+(+1,-0.2);
                     \node(a)at(9.3,-3){   $\Delta_3$};
                  
             \draw[thin,>-stealth,->](6,-1)--+(-1,+0.2);
          \draw[thin,>-stealth,->](6,-1)--+(-1,-0.2);
           \node(a)at(4.7,-1){   $\Delta_4$};

 \draw[thin,>-stealth,->](7,0)--+(-0.2,1.1);
 \node(a)at(7.2,1){   $\Delta_6$};
                   
                   \draw[fill=white ] (7,-4)circle(2pt);
           \draw[fill=white] (7,0)circle(2pt);
               \draw[fill=white   ] (7,-2)circle(2pt);
                  \draw[fill=white   ] (6,-1)circle(2pt);
                        \draw[fill=white   ] (8.5,0)circle(2pt);
                         \draw[fill=white   ] (8,-3)circle(2pt);

\node(a)at(6,-0.7){ $\it 3$};
\node(a)at(6.6,0){ $\it 5/2$};
 
\node(a)at(6.7,-2.1){ $\it 2$};
\node(a)at(6.7,-3.8){ $\it 1$};
\node(a)at(7.7,-2.8){ $\it 2$};
  \end{tikzpicture} 
  \end{center}
   
 \end{example}   
 
}

 \section{Inner rates on a minimal surface} \label{sec:inner rates}

Let $\eta \colon Y \to \C^2$  be a sequence of blow-ups of points starting with the blow-up of the origin of $\C^2$ and let $C$ be an irreducible component of $\eta ^{-1}(0)$.  Let $(\gamma_1,0)$ and $(\gamma_2,0)$ be  two irreducible curve germs whose strict transforms by $\eta $ meet $C$ at two distinct points which are smooth points of $\eta ^{-1}(0)$. Then the contact $q_{inn}(\gamma_1, \gamma_2)= q_{out}(\gamma_1, \gamma_2) $  of $\gamma_1$ and $\gamma_2$ in $\C^2$ does not depend on the choice of  $\gamma_1$ and $\gamma_2$.  

\begin{definition}\label{def:inner rate} We call $ q_{inn}(\gamma_1,
  \gamma_2)$ the {\it  inner rate} of $C$  and we denote it by
  $q_C$. 
  
  If $\gamma$ is a test curve at a vertex $(u)$ of  $\rho_\ell$ for some generic projection $\ell \colon (X,0) \to (\C^2,0)$ of a normal surface $(X,0)$,  we will say that $q_{C_u}$ is the {\it inner rate of $\gamma$} and denote $q_{\gamma}=q_{C_u} = q_u$. 
\end{definition}

\begin{example}\label{example:inner rates} If $(X,0)$ is a minimal singularity, then for each vertex $(u)$ of $T'_0$, we have $q_u = \widehat{s}(u)$ where $\widehat{s} \colon T'_0 \to \frac{1}{2} \N^*$ is defined before the statement of Corollary \ref{cor:graph map}.
\end{example}
 
\begin{lemma}[{\cite[Lemma 15.1]{NPP}}] \label{rk:inner rate} Let $\pi \colon X' \to X$ be a resolution of $X$ and let  $E$ be an irreducible component of the exceptional divisor $\pi^{-1}(0)$. Let $\gamma$ and $\gamma'$ be two complex curve germs in $(X,0)$ whose strict transforms by $\pi$ are curvettes of $E$ meeting $E$ at two distinct points.  Then  $q_{inn}(\gamma, \gamma')$  is independent of the choice of $\gamma$ and $\gamma'$.  Moreover, if $\ell \colon (X,0)\to (\C^2,0)$ is a generic projection which is also generic for the curve $\gamma \cup \gamma'$ and if $\eta$ is a resolution of the curve $\ell(\gamma) \cup \ell(\gamma')$, then there is a component  $C$ of $\eta^{-1}(0)$ such that  $\ell(\gamma)^*$ and $\ell(\gamma')^*$  are curvettes of $C$, and we have  $q_C = q_{inn}(\gamma, \gamma')$. 
\end{lemma}
  
\begin{definition}  \label{def:inner rate2} We set $q_E = q_{inn}(\gamma, \gamma')$ and we call $q_E$ the inner rate of $E$. 
\end{definition}

\begin{proposition}[{\cite[Proposition 15.3]{NPP}}] \label{lem:inner contact computation} Let $\gamma$ and $\gamma'$ be two complex curves on $(X,0)$.  Consider a  resolution  $\pi \colon X' \to X$    which factors through the Nash modification and through the blow-up of the maximal ideal  and which is a resolution of the complex curve $\gamma \cup \gamma'$ and set $\pi^{-1}(0) = \bigcup_{v}E_v$.  Let $\widetilde{G}$ be the resolution graph of $\pi$ whose vertices $(v)$ are weighted by the inner rates introduced in Definition \ref{def:inner rate2}. Let  $(v)$ and $(v')$ be the vertices of $G$ such that $\gamma^* \cap E_{v} \neq \emptyset$ and    ${\gamma'}^* \cap E_{v'} \neq \emptyset$.  Then $q_{inn}(\gamma, \gamma')= q_{v,v'}$ where   $q_{v,v'}$ is the   maximum among minimum of inner rates  along paths from $(v)$ to $(v')$ in the graph $\widetilde{G}$. 
\end{proposition}

 In \cite{NPP}, we introduced a resolution  $\mu_0 \colon W_0 \to X$,  the so-called LNE$_{test}$-resolution, which is a good resolution for every principal component over every  test curve for every generic projection $\ell$, i.e., the strict transform of each such principal component is a curvette of an irreducible component of $\mu_0^{-1}(0)$. We will now explicitly describe the graph of the LNE$_{test}$-resolution of a minimal resolution and the inner rate  attached to each vertex. This will be a key tool in the proof of the ``if" direction of Theorem  \ref{theorem:main}.

Let us first recall the definition of $\mu_0$. Consider a  generic projection $\ell \colon (X,0) \to (\C^2,0)$. Let  $X_{\ell}$   be the pull-back  of 
  $\ell$ and  $\rho_{\ell} \colon Z_{\ell} \to \C^2$  (cf.\ Definition \ref{def:test curve}) and let   $\alpha_{\ell} \colon X'_{\ell}  \to X_{\ell}$ be the minimal  good resolution of $X_{\ell}$. This induces a resolution $\pi_{\ell} \colon X'_{\ell} \to X$   which  factors through $X_0$ and  a projection  $\widetilde{\ell} \colon X'_{\ell} \to Z_{\ell}$. Let $\xi_{\ell} \colon X'_{\ell}  \to W_{\ell}$  be the morphism obtained by blowing down iteratively the exceptional $(-1)$-curves which  are not on simple paths joining vertices of $G_0'$   (Definition \ref{def:G'}). We then obtain a good resolution  $\mu_{\ell} \colon W_{\ell} \to X$ of $(X,0)$ which factors through  $\pi_0 \colon X_0 \to X$ by a morphism $\beta_{\ell} \colon W_{\ell} \to X_0$. By \cite[Lemma 13.1]{NPP}, the morphism $\beta_{\ell}$ does not depend on $\ell$.  We set $\beta_0 = \beta_{\ell}$, $W_0 = W_{\ell}$ and $\mu_0 = \mu_{\ell}$.

\begin{definition} \label{def:LNE test resolution} We call $\mu_0 \colon W_0 \to X$ the LNE$_{test}$-resolution of $(X,0)$. We denote by $\Gamma_0$ the graph of $\mu_0$ and by $\Gamma'_0$ the subgraph of $\Gamma_0$ which consists of the union of all simple paths joining $\cal L$- or $\cal P$-nodes. 
\end{definition}

In the case of a minimal  singularity, we have $G_0 = G'_0$ (Remark \ref{rk:G_0=G'_0}) and then $\Gamma_0 = \Gamma'_0$. The following proposition describes explicitly the graph $\Gamma_0$ with inner rates  from the minimal resolution graph  $G$ for any minimal singularity.

\begin{proposition}\label{innergeomertythm} Let $(X,0)$ be a minimal surface singularity. 
\begin{enumerate}
\item[(1)] The inner rates of the vertices of $G_0$ are determined by $G$ as follows: 
\begin{enumerate}
\item If $(v)$ is a vertex of $G$, then $q_v = s(v)$;
\item If $(v)$ is a vertex obtained by blowing up a double point corresponding to a central edge $(v_1)-(v_2)$ of $G$, then $q_v = s(v_i) +1/2$.
\end{enumerate}
\item[(2)] The graph $\Gamma_0$ is obtained by performing the following blow-ups for each edge $(u_1)-(u_2)$ joining two $\Delta$-nodes of $T_0$.  After exchanging $u_1$ and  $u_2$ if necessary  we can assume  $\widehat{s}(u_1) < \widehat{s}(u_2)$. So $\widehat{s}(u_1) \in \N^*$. {Let $\widehat{\tau}$ be the connected component of $T'_0(\widehat{s}>\widehat{s}(u_1))$ containing $(u_2)$ and let $\tau$ be the connected component of $G_0(s > \widehat{s}(u_1))$ such that $L(\tau)=\widehat{\tau}$ (see Corollary \ref{cor:graph map}).  }
\begin{enumerate}
\item Either {$\widehat{s}(u_2) \in \N^*$, i.e., $\widehat{s}(u_2)=\widehat{s}(u_1)+1$}.  Then  we blow up all double points corresponding to edges  $(w_1)-(w_2)$ in $\tau$ such that  $s(w_1)={\widehat{s}(u_1)}$ and $s(w_2)={\widehat{s}(u_2)}$; each created vertex $(w)$ has inner rate $q_{w} = s(v_1)+1/2$.

 \item Otherwise,  {$\widehat{s}({u_2}) \in \N^*+1/2$, i.e., $\widehat{s}({u_2})  = \widehat{s}(u_1)+1/2$.}  Then, we first blow up the double points corresponding to  edges $(w_1)-(w'_2)$  in $G$ such that ${(w'_2) } \in V(\tau)$,  $s(w_1)={\widehat{s}(u_1)}$ and $s(w'_2)={\widehat{s}(u_1)+1}$;  each created vertex $(w_2)$ has inner rate $q_{w_2} = {\widehat{s}(u_1)}+1/2$. Then we blow up the double point corresponding to each created edge $(w_1)-(w_2)$. Each created vertex $(w)$ has inner rate $q_{w} = {\widehat{s}(u_1)}+1/3$.
\end{enumerate}
\end{enumerate}
\end{proposition}

\begin{remark} \label{rk:extend L} Let $T$ be the dual graph of $\rho_{\ell}$ and let $T'$ be the subgraph of $T$ which consists of the union of all simple paths joining $\Delta$-nodes to the root vertex in $T$ (see Definition \ref{def:test curve}).  So with the notations introduced before  Corollary \ref{cor:graph map}, we have $V(T_0) \subset V(T)$ and $V(T'_0) \subset V(T')$.   By construction, the graph-map $L \colon G_0 \to T'_0$ extends to a graph-map $L \colon \Gamma_0 \to T'$ such that for every vertex $(v)$ of $\Gamma_0$, we have $q_v = q_{L(v)}$.   \end{remark}

\begin{example} \label{example} We consider again the minimal singularity of Example \ref{example0}. The two graphs below are the graph $\Gamma_0$ and the dual graph  $T$ of $\rho_{\ell}^{-1}(0)$ with  vertices weighted by the corresponding inner rates. The black vertices are the vertices of $T$ which are not in $T_0$ (resp. the vertices of $\Gamma_0$ which are not in $G_0$). The images of vertices and edges of $\Gamma_0$ are sent horizontally on that of $T'$ by the extended graph-map $L \colon \Gamma_0 \to T'$ introduced in Remark \ref{rk:extend L}.

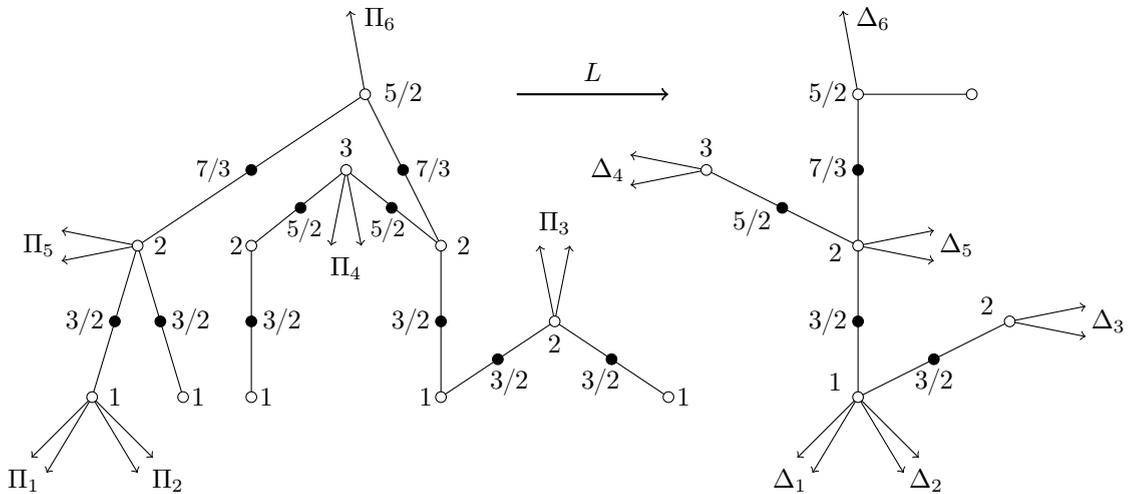
 \begin{center}
\begin{tikzpicture}
 \begin{scope}[xshift=-1.5cm]
   \draw[thin ](2,0)--(3,-2);
     \draw[thin ](2,0)--(-1,-2);
       \draw[thin ](1.75,-1)--(0.5,-2);
       \draw[thin ](1.75,-1)--(3,-2);
          \draw[thin ](3,-2)--(3,-4);
           \draw[thin ](4.5,-3)--(3,-4);
             \draw[thin ](4.5,-3)--(6,-4);
               \draw[thin ](0.5,-2)--(0.5,-4);
                 \draw[thin ](-1,-2)--(-0.4,-4);
                 \draw[thin ](-1,-2)--(-1.6,-4);
                 
\draw[thick, >-stealth,-> ](4,0)--(6,0);
 \node(a)at(5,0.3){   $L$};
                 
  \draw[thin,>-stealth,->](-1.6,-4)--+(+0.8,-0.8);
       \draw[thin,>-stealth,->](-1.6,-4)--+(-0.8,-0.8);
        \draw[thin,>-stealth,->](-1.6,-4)--+(-0.6,-1);
          \draw[thin,>-stealth,->](-1.6,-4)--+(+0.6,-1);
            \node(a)at(-2.5,-5.1){   $\Pi_1$};
       \node(a)at(-0.6,-5.1){   $\Pi_2$};
       
         \draw[thin,>-stealth,->](1.75,-1)--+(+0.2,-1);
       \draw[thin,>-stealth,->](1.75,-1)--+(-0.2,-1);
        \node(a)at(1.75,-2.3){   $\Pi_4$};
        
          \draw[thin,>-stealth,->](4.5,-3)--+(+0.2,1);
       \draw[thin,>-stealth,->](4.5,-3)--+(-0.2,1);
        \node(a)at(4.5,-1.7){   $\Pi_3$};
        
         \draw[thin,>-stealth,->](2,0)--+(-0.2,1.1);
 \node(a)at(2.2,1){   $\Pi_6$};
 
  \draw[thin,>-stealth,->](-1,-2)--+(-1,0.2);
             \draw[thin,>-stealth,->](-1,-2)--+(-1,-0.2);
               \node(a)at(-2.3,-2){   $\Pi_5$};

        \draw[fill=white] (2,0)circle(2pt);
      \draw[fill=white] (-1,-2)circle(2pt);
        \draw[fill=white] (0.5,-2)circle(2pt);
          \draw[fill=white] (3,-2)circle(2pt);
            \draw[fill=white] (4.5,-3)circle(2pt);
             \draw[fill=white] (3,-4)circle(2pt);
              \draw[fill=white] (6,-4)circle(2pt);
                   \draw[fill=white] (1.75,-1)circle(2pt);
                    \draw[fill=white] (0.5,-4)circle(2pt);
                      \draw[fill=white] (-0.4,-4)circle(2pt);
                       \draw[fill=white] (-1.6,-4)circle(2pt);
\draw[fill=black] (5.25,-3.5)circle(2pt);
\draw[fill=black] (3.75,-3.5)circle(2pt);
\draw[fill=black] (3,-3)circle(2pt);
\draw[fill=black] (0.5,-3)circle(2pt);
\draw[fill=black] (2.35,-1.5)circle(2pt);
\draw[fill=black] (1.15,-1.5)circle(2pt);
\draw[fill=black] (-0.7,-3)circle(2pt);
\draw[fill=black] (-1.3,-3)circle(2pt);

\draw[fill=black] (2.5,-1)circle(2pt);
\draw[fill=black] (0.5,-1)circle(2pt);

\node(a)at(2.5,0){ $\it 5/2$};
\node(a)at(1.75,-0.7){ $\it 3$};
\node(a)at(3.3,-2){ $\it 2$};
\node(a)at(0.3,-2){ $\it 2$};
\node(a)at(-0.7,-2){ $\it 2$};
\node(a)at(4.5,-3.3){ $\it 2$};

\node(a)at(6.2,-4){ $\it 1$};
\node(a)at(2.8,-4){ $\it 1$};
\node(a)at(0.7,-4){ $\it 1$};
\node(a)at(-0.2,-4){ $\it 1$};
\node(a)at(-1.3,-4){ $\it 1$};

\node(a)at(5.1,-3.8){ $\it 3/2$};
\node(a)at(3.9,-3.8){ $\it 3/2$};
\node(a)at(0.9,-3){ $\it 3/2$};
\node(a)at(2.6,-3){ $\it 3/2$};

\node(a)at(-0.3,-3){ $\it 3/2$};
\node(a)at(-1.7,-3){ $\it 3/2$};

\node(a)at(2.3,-1.8){ \small{$\it 5/2$}};
\node(a)at(1.2,-1.8){ \small{$\it 5/2$}};

\node(a)at(2.9,-1){ \small{$\it 7/3$}};
\node(a)at(0,-1){ \small{$\it 7/3$}};

\end{scope}


  \draw[thin ](7,0)--(7,-4);
   \draw[thin ](7,0)--(8.5,0);
    \draw[thin ](7,-2)--(5,-1);
     \draw[thin ](7,-4)--(9,-3);

      \draw[thin,>-stealth,->](7,-4)--+(+0.8,-0.8);
       \draw[thin,>-stealth,->](7,-4)--+(-0.8,-0.8);
        \draw[thin,>-stealth,->](7,-4)--+(-0.6,-1);
          \draw[thin,>-stealth,->](7,-4)--+(+0.6,-1);
            \node(a)at(6.1,-5.1){   $\Delta_1$};
       \node(a)at(7.9,-5.1){   $\Delta_2$};

             \draw[thin,>-stealth,->](7,-2)--+(+1,0.2);
             \draw[thin,>-stealth,->](7,-2)--+(+1,-0.2);
               \node(a)at(8.3,-2){   $\Delta_5$};
               
         \draw[thin,>-stealth,->](9,-3)--+(+1,0.2);
                  \draw[thin,>-stealth,->](9,-3)--+(+1,-0.2);
                     \node(a)at(10.3,-3){   $\Delta_3$};
                  
             \draw[thin,>-stealth,->](5,-1)--+(-1,+0.2);
          \draw[thin,>-stealth,->](5,-1)--+(-1,-0.2);
           \node(a)at(3.7,-1){   $\Delta_4$};

 \draw[thin,>-stealth,->](7,0)--+(-0.2,1.1);
 \node(a)at(7.2,1){   $\Delta_6$};
                   
                   \draw[fill=white ] (7,-4)circle(2pt);
           \draw[fill=white] (7,0)circle(2pt);
               \draw[fill=white   ] (7,-2)circle(2pt);
                  \draw[fill=white   ] (5,-1)circle(2pt);
                        \draw[fill=white   ] (8.5,0)circle(2pt);
                         \draw[fill=white   ] (9,-3)circle(2pt);
                         
\draw[fill=black] (7,-3)circle(2pt);
\draw[fill=black] (7,-1)circle(2pt);
\draw[fill=black] (8,-3.5)circle(2pt);
\draw[fill=black] (6,-1.5)circle(2pt);

\node(a)at(5,-0.7){ $\it 3$};
\node(a)at(6.6,0){ $\it 5/2$};
 
\node(a)at(6.7,-2.1){ $\it 2$};
\node(a)at(6.7,-3.8){ $\it 1$};
\node(a)at(8.7,-2.8){ $\it 2$};
\node(a)at(8,-3.8){ $\it 3/2$};
\node(a)at(6.6,-3){ $\it 3/2$};
\node(a)at(6.6,-1){ $\it 7/3$};
\node(a)at(5.6,-1.7){ $\it 5/2$};
  \end{tikzpicture} 
  \end{center}
\end{example}

\begin{proof} By construction, $\mu_0$ is the minimal resolution of $(X,0)$ such that for every generic projection $\ell \colon (X,0) \to (\C^2,0)$, for every vertex $(v)$ of $T'$ and every test curve $\gamma$ at $(v)$, a component   $\widehat{\gamma}$ of $\ell^{-1}(\gamma)$ is principal if and only if its strict transform by $\mu_0$ is a curvette of a component $E_v$ of $\mu_0^{-1}(0)$ such that $(v) \in V(\Gamma'_0)$. Then, in order to describe $\Gamma_0$, it suffices to start from the resolution $\pi_0 \colon X_0 \to X$ and to blow up points iteratively  until resolving the principal components over test curves at vertices of $T'$.  

 Let us first describe $T$ and the inner rates of vertices of the subtree $T'$. By  Theorem  \ref{prop:discriminant}, $T_0$ is the minimal resolution tree of $\Delta$, and $\Delta$ consists of $A_n$-curves having integral contacts between each other. Therefore, a vertex $(u)$  of $T'_0$ has inner rate in  $\N+1/2$. Now, $T$ is obtained from $T_0$ by blowing up every edge between two adjacent $\Delta$-nodes in $T'_0$. Either the two nodes have integral inner rates  $n$ and $n+1$, then the inner rate of the created vertex is $n+1/2$, or one of the inner rates is $n$ and the other $n+1/2$, in which case the inner rate of the created vertex equals $n+1/3$. Summarizing, a vertex $(u)$ of $T'$ has one of the following  types (1a), (1b), (2a) or (2b):

\noindent
{\bf Case (1)} $u \in V(T'_0)$ and we have two cases for the inner rate $q_u$: 
\begin{description}
\item[(1a)]   $q_u=n$ with $n \in \N^*$;
\item[({1b})]  $q_u=n+1/2$ with $n \in \N^*$.   
\end{description}

\noindent
{\bf Case (2)}  $u \not\in V(T'_0)$, i.e.,  $(u)$ is obtained by blowing up the edge between two $\Delta$-nodes  $(u_1)$ and $(u_2)$ of $T_0$, and maybe after exchanging $u_1$ and $u_2$, we  are in one of the following two cases: 
\begin{description}
\item[(2a)]   $q_{u_1} = n$ and $q_{u_1} = n+1$ with $n \in \N^*$, in which case $q_u = n+1/2$;
 \item[(2b)]   $q_{u_1} = n$ and $q_{u_1} = n+1/2$ with $n \in \N^*$, in which case $q_u = n+1/3$.
\end{description}

We will show that the four cases in the statement of Proposition \ref{innergeomertythm} correspond to the principal components over test curves in the four cases  just described. 

 Let $(u)$ be a vertex of type $(1a)$ and let $\gamma$ be a test curve at $(u)$. Set $\widehat{s}(u)=n$. We have $q_u=\widehat{s}(u)=n$ (see Example \ref{example:inner rates}).  By  \ref{cor:(i)}  of Corollary \ref{cor:graph map},  every principal component  of $\ell^{-1}(\gamma)$ is the $\rho_{\ell}$-image of a curvette of a component $E_v$ of $G$ such that $s(v)=n$, and  since $q_u=q_v$, we then have  $q_v=s(v)$.

Assume now that  $u$ is of type $(1b)$, so we have $q_u=n+1/2$ with $n \in \N^*$.  By  \ref{cor:(ii)} of Corollary \ref{cor:graph map},  $L^{-1}(u)$ contains a unique vertex $(v)$ which is obtained by blowing up a central edge $(v_1)-(v_2)$ of $G$.  Therefore the principal components of $\ell^{-1}(\gamma)$  are the $\rho_{\ell}$-images of  curvettes  of  $E_v$ (in fact, there are exactly two principal components in $\ell^{-1}(\gamma)$). We have  $q_v=q_u=n+1/2$.  On the other hand,  we have $s(v_1)=s(v_2)=n$. Therefore, we obtain $q_v = s(v_1)+1/2$.

Assume now that $(u)$ is of type $(2a)$ so $C_u$ is obtained by blowing up the intersection point $q=C_{u_1} \cap C_{u_2}$  between two $\Delta$-curves $C_{u_i}$  of $\rho^{-1}(0)$ such that  $q_{u_1} =n$ and $q_{u_2} = n+1$ for some integer $n \in \N^*$. Let $\gamma$ be a test curve at $(u)$. Then the principal components of $\ell^{-1}(\gamma)$ have their strict transforms by $\pi$ passing through the points $p$ such that $\ell_n(p)=q$. By Corollary \ref{cor:graph map}, each such $p$ is an intersection point $p=E_{w_1}\cap E_{w_2}$ where $(w_1)$ and $(w_2)$ are two vertices of $G$ such that  $s(w_1)=n$ and $s(w_2)=n+1$  and {$(w_2)$} is in the connected component 
{$\tau$} of $G_0({s}> \widehat{s}(u_1))$ {introduced in the statement of the proposition}.

We now have to prove that the inner rate $q_w$ of the vertex obtained by blowing up the edge $(w_1)-(w_2)$ equals $n+1/2$. 

Let us  first prove that the germ of {the} morphism $\ell_n
\colon (\widetilde{X},p) \to (Y_n,q)$ is an isomorphism. We  choose
coordinates  $(s,t)$ in $\widetilde{X}$ centered at $p$ and
coordinates  $(s',t')$ in $Y_n$ centered at $q$ such that
$\ell_n(s,t)=(s^{\alpha},t^{\beta})$. Assume $\ell = (x,y)$ where
$h=x$ is a generic linear form on $(X,0)$ {and $h$ is the
  composition  $h'\circ\ell$ where $h'\colon \C^2\to\C$ is {the
    projection to} the first coordinate}. Since $(X,0)$ is minimal, the multiplicity of $h$ along $E_{w_i}$ equals $1$, so $(h \circ \pi)(s,t) = s t \zeta(s,t)$ where $\zeta$ is a unit in $\C\{s,t\}$. Let $\rho =  \rho_1 \circ \cdots \circ \rho_n$ be as defined in Theorem \ref{prop:discriminant}. Since the multiplicities of $x$ along $C_{u_1}$ and $C_{u_2}$ also equal $1$, then $({h'} \circ \rho)(s',t') = s't' \delta(s',t')$  where  $\delta$ is a unit in $\C\{s',t'\}$. Since $\rho \circ \ell_n = \ell \circ \pi$, we then obtain 
$st \zeta(s,t) = s^{\alpha} t^{\beta} \delta(s^{\alpha}, t^{\beta})$. This implies $\alpha = \beta = 1$, i.e., $\ell_n \colon (\widetilde{X},p) \to (Y_n,q)$ is the germ of an  isomorphism. 

Let   $e_p \colon X_p \to  \widetilde{X}$ be the blow-up of $p$ and consider the  $\Bbb P^1$-curve  $E_w= e_p^{-1}(p)$.  Let $e_q \colon Y_q \to Y_n$ be the blow-up of $q$,   so we have  $C_u = e_q^{-1}(q)$.  We then have an induced isomorphism $\ell'_n \colon N(E_w) \to N(C_u)$  from a neighbourhood $N(E_w)$ of $E_v$ to a neighbourhood $N(C_u)$ of  $C_u$.  Therefore $q_w=q_u=  n+1/2$. 

Let us now treat Case (2b) so $C_u$ is obtained by blowing up the intersection point $q'=C_{u_1} \cap C_{u_2}$  between a $\Delta$-curve $C_{u_1}$ of $ \rho^{-1}(0)$ such that  $q_{u_1} =n$ and a $\Delta$-curve $C_{u_2}$ with $q_{u_2} =n+1/2$.   By Corollary \ref{cor:graph map},  the principal components of $\ell^{-1}(\gamma)$ have their strict transforms by $\pi_0$ passing through the points $p=E_{w_1}\cap E_{w_2}$ where {$(w_2)$} is in the connected component {$\tau$} of $G_0({s}> s(v_1))$,  $s(w_1)=n$ and  $s(w_2)=n+1/2$.  Set $q=C_{u_1} \cap C_u$. Then $\ell_n$ induces a morphism $\ell'_n \colon (X_0,p) \to (Y_{\ell},q)$.

Let us prove that $\ell'_n \colon (X_0,p) \to (Y_{\ell},q)$ is the germ of an isomorphism. We  choose coordinates  $(s,t)$ in $X_0$ centered at $p$ and  coordinates  $(s',t')$ in $Y_{\ell}$ centered at $q$ such that $\ell_n(s,t)=(s^{\alpha},t^{\beta})$. Using again the notations of Case (2a),  the multiplicity of $h$ along $E_{w_1}$ equals $1$. Since the strict transform of $h$ does not pass through $E_{w_1} \cap E_{w_2}$, the strict transform of $h$ along $E_{w_2}$ equals $2$. Then  $(h \circ \pi_0)(s,t) = s t^2 \zeta(s,t)$ where $\zeta$ is a unit in $\C\{s,t\}$.   Since the multiplicities of $x$ along $C_{u_1}$ and $C_{u_2}$  equal  respectively $1$ and $2$, then $({h'} \circ \rho_{\ell})(s',t') = s'(t')^2 \delta(s',t')$  where  $\delta$ is a unit in $\C\{s',t'\}$ {and $h'$ is as before}. Since $\rho_{\ell} \circ \ell'_n = \ell \circ \pi_0$ on the germ $(X_0,p)$, we then obtain $st^2 \zeta(s,t) = s^{\alpha} t^{2\beta} \delta(s^{\alpha}, t^{2\beta})$. This implies $\alpha = \beta = 1$, i.e., $\ell'_n \colon (X_0,p) \to (Y_{\ell},q)$ is the germ of an  isomorphism. 

We now have to prove that the inner rate $q_w$ of the vertex obtained by blowing up the edge $(w_1)-(w_2)$ equals $n+1/3$.  Let   $e_p \colon X_p \to  X_0$ be the blow-up of $p$ and consider the  $\Bbb P^1$-curve  $E_w= e_p^{-1}(p)$.  Let $e_q \colon Y_q \to Y_{\ell}$ be the blow-up of $q$,   so    $C_u = e_q^{-1}(q)$.  We then have an induced isomorphism $\ell''_n \colon N(E_w) \to N(C_u)$  from a neighbourhood $N(E_w)$ of $E_v$ to a neighbourhood $N(C_u)$ of  $C_u$.  Therefore,   $q_w=q_u=  n+1/3$.
\end{proof}

\section{Minimal implies LNE} \label{sec: main proof}

In this Section, we prove the  ``if'' direction of Theorem \ref{theorem:main}]. We start by proving two preliminary results.
\subsection{Blow-up of a minimal singularity} \label{subsec: blow-up minimal}

The following result is  \cite[Th\'eor\`eme 5.9]{BL}. The authors prove it there without using the existence of a resolution of $(X,0)$. We give here a short proof using this fact.

{\begin{proposition}  \label{lem:generic2}   Let  $(X,0)$ be a minimal surface singularity and let $e' \colon X' \to X$ be the blow-up of the origin. Then $X'$ is normal and for every  singular point $p \in X'$,  $(X',p)$ is a minimal singularity. 
 \end{proposition}

\begin{proof}[of Proposition \ref{lem:generic2}]  Since  $(X,0)$ is minimal, then $(X,0)$ is a rational singularity and then,  its blow-up is normal (\cite{Tyu}).

Let $\pi \colon  \widetilde{X}  \to X$ be the minimal resolution of $(X,0)$ and let
$G$ be its resolution graph. Since $(X,0)$ is rational, then
$\pi$ factors through the blow-up of the maximal ideal 
(\cite{A}). Assume $(X',p)$ is not smooth. Then $(X',p)$ has minimal resolution graph  one of the connected components
$\cal G$ of $G$ minus the $\cal L$-nodes. So $\cal G$ is a rational graph and $(X',p)$ is rational. Moreover, since the minimal cycle of $(X,0)$ is reduced, the minimal cycle of $(X',p)$ is  also reduced. Then, by Proposition \ref{prop:minimal iff rational and reduced fundamental cycle}, $(X',p)$ is  a minimal surface singularity. 
\end{proof}}

\subsection{Outer contact between curves and blow-up} \label{subsec: outer contact and blow-up}
\begin{proposition} \label{prop:formula}
Let $(\gamma_1,0)$  and $(\gamma_2 ,0)$  be two complex curve germs in $ (\C^N,0)$.  Let
$\widetilde{e}:(Z,E)\to (\C^N,0)$ be the blow-up of the origin and let
$\gamma_i^*$ be the strict transform of $\gamma_i$ by $\widetilde{e}$. Assume that
$q_{out}(\gamma_1,\gamma_2)\geq 2$. Then $\gamma_1^*\cap\gamma_2^* =p\neq
\emptyset$ and 
$
q_{out}(\gamma_1^*,\gamma_2^*)= q_{out}(\gamma_1,\gamma_2)-1
$, where $q_{out}(\gamma_1^*,\gamma_2^*)$ means outer contact in the germ $(Z,p)$. \end{proposition}

\begin{proof} Let $m_i$ be the multiplicity of $\gamma_i$ for $i=1,2$.  Set $q_o = q_{out}(\gamma_1,\gamma_2)$. We can choose
coordinates $(z_1,\ldots,z_N)$ for $\C^N$ and Puiseux 
parametrizations of $\gamma_i$ as follows:
\begin{align}
\gamma_i \colon t \in \C \mapsto \big(t,\gamma_{i,2}(t),\dots, \gamma_{i,N}(t)\big)
\end{align}
where $\gamma_{i,j}(t) \in \C\{t^{\frac{1}{m_i}}\}$ are fractional power series whose terms with degrees $ < q_o$ all  coincide while there are at least two coefficients of $t^{q_0}$ which differ.

 Let us express $\widetilde{e}$ in the chart $U_1\subset Z$ over $z_1 \neq 0$. In the corresponding local coordinates $(z_1,u_2,\ldots,u_N)$ of  $U_1\cong \C^N$, we have   ${\widetilde{e}}(z_1,u_2,\dots,u_N) =
(z_1,z_1u_2,\dots, z_1u_N)$ and the exceptional divisor
$E={\widetilde{e}}^{-1}(0)$ has equation $z_1=0$. In this chart,  the strict transform $\gamma_i^*$ is parametrized by: 
$\gamma_i^*(t)=(t, \gamma_{i,2}^*(t), \dots,
\gamma_{i,N}^*(t))$  with  $\gamma_{i,j}(t)=t\gamma_{i,j}^*(t)$ for all 
$j>1$. This implies: 
\begin{align*}
\norm{\gamma_1(t)-\gamma_2(t)} &=\norm{\big(t-t,t\gamma_{1,2}^*(t)
  -t\gamma_{2,2}^*(t),\dots, t\gamma_{1,N}^*(t) -t\gamma_{2,N}^*(t)\big)} \\
&=\lvert t \rvert \norm{\gamma_1^*(t) -\gamma_2^*(t)}. 
\end{align*}
Since $q_o\geq 2$, 
then all $\gamma_{i,j}^*(t)$ only have terms of degree $\geq
 1$, and hence $\norm{\gamma_i^*(t)} = \Theta(\vert t \vert)$. Therefore  $
 \norm{\gamma_1^*(t)  -\gamma_2^*(t)}) =
 \Theta(t^{q_{out}(\gamma_1^*,\gamma_2^*)})$ and  then  $\norm{\gamma_1(t)-\gamma_2(t)}  =\Theta(|t|^{ q_{out}(\gamma_1^*,\gamma_2^*) +1})$. On the other hand, we have  
 $\norm{\gamma_1(t)-\gamma_2(t)}  = \Theta(|t|^{q_o})$. Hence 
  $q_o = 1 + q_{out}(\gamma_1^*,\gamma_2^*)$. 
\end{proof}

\begin{proof}[of the  ``if'' direction of Theorem \ref{theorem:main}]

Let $(X,0)$ be a minimal surface singularity with generic   projection
$\ell \colon (X,0)\to (\C^2,0)$, let $\Pi$ be its polar curve and let $\Delta = \ell(\Pi)$ be its discriminant curve.  {We use again the notations of Section \ref{sec:characterization LNE}. Let  $\rho_{\ell} \colon  {Z}_{\ell} \to \C^2$ be the  sequence of blow-ups of points introduced in Section \ref{sec:characterization LNE}.} We also use the notations introduced in the previous Section, in particular,  $\Gamma_0$ is the dual graph of the LNE$_{test}$-resolution  of $(X,0)$ (Definition \ref{def:LNE test resolution}), and $G_0$ is the graph of the resolution $\pi_0 \colon X_0 \to X$.  

We have to check Conditions  ($1^*$) and ($2^*$) 
of Theorem  \ref{cor:complex characterization of normal embedding} for any 
test curve $(\gamma,0) \subset (\C^2,0)$ (Definition \ref{def:test curve}).  {We   have to consider the following cases for the values of the inner rate $q_{\gamma}$ (see proof of Proposition \ref{innergeomertythm}):
\begin{itemize}
\item Case (1{a}).  $q_\gamma$   is an integer $n$ (in particular, $\gamma$ is smooth); 
\item Cases (1b) and (2a).  $q_\gamma=n+1/2$ with $n \geq 1$ an integer; 
\item  Case (2b).   $q_\gamma =  n +1/3$ with $n \geq 1$ an integer.
  \end{itemize}}

Assume first that $q_{\gamma}=1$. Then $\gamma$ is a generic line through the origin of $\C^2$, so  $(\ell^{-1}(\gamma),0)$ is a generic hyperplane section of $(X,0)$.  Since $(X,0)$ is minimal,  the generic hyperplane section $(\ell^{-1}(\gamma),0)$  also has a minimal  singularity  (\cite[Lemma 3.4.3]{K}) so  it  is a union of $m(X,0)$ smooth transversal curves, where $m(X,0)$ denotes the multiplicity of $(X,0)$. {Then the multiplicity of every component of   $(\ell^{-1}(\gamma),0)$ equals $1$}, and if $\gamma_1$ and $\gamma_2$ are two components of $(\ell^{-1}(\gamma),0)$, then $q_{inn}(\gamma_1,\gamma_2) = q_{out}(\gamma_1,\gamma_2)=1$. Therefore   Conditions ($1^*$) and ($2^*$) of Theorem \ref{cor:complex characterization of normal embedding} are satisfied. 

We then have to prove Conditions ($1^*$) and ($2^*$)  for any test curve which is not a curvette of the root vertex. 
\vskip0,2cm\noindent
 Let us first prove that any test curve satisfies Condition  ($1^*$).

In case (1a),  $\gamma$ is a smooth curve, so $mult(\gamma)=1$. On the
other hand,  $q_{\gamma} \in \N$ means that $E$ is a component of the
irreducible divisor of the minimal resolution $\pi \colon
\widetilde{X} \to X$ of $X$. Since $(X,0)$ is minimal, its minimal
cycle is reduced. The maximal cycle is the compact part of the total
transform by $\pi$ of a generic linear function $h \colon (X,0) \to
(\C,0)$.  Then in particular, {since for rational and hence
  minimal singularities the maximal cycle and the minimal cycle agrees}, we have $mult_E(h)=1$, which means that  the multiplicity of any curvette of $E$ is $1$. Therefore $mult(\widehat{\gamma}) = 1$.

In cases (1b) and (2a), we have $mult(\gamma)=2$. On the other hand, $E$ is obtained by blowing up the  intersection  point $p$ of two exceptional curves $E_{v_1}$ and  $E_{v_2}$  corresponding to {vertices} of the graph of the minimal resolution of $X$ {(see the proof of Proposition \ref{innergeomertythm} for Case (2a))}. Since  $Z_{min}$ is reduced, then $mult_{E_{v_1}}(h) = mult_{E_{v_2}}(h) =1$. Since the strict transform of $h$ does not pass through $p$, then $mult_E(h) = mult_{E_{v_1}}(h) + mult_{E_{v_2}}(h) = 2$, which means that $mult(\widehat{\gamma}) = 2$. 

In case (2b), the multiplicity of $\gamma$ equals $3$, since $C$ is obtained by blowing up the intersection point between two exceptional components, one along which a generic linear form on $(\C^2,0)$ has multiplicity $1$  and the other $2$. On the other hand, $E$ is obtained by blowing up the  intersection point $p$ of two exceptional curves $E_{v_1}$ and  $E_{v_2}$ such that  $mult_{E_{v_1}}(h)=1$  and  $mult_{E_{v_1}}(h)=2$ {(see again the proof of Proposition \ref{innergeomertythm})}. Therefore  $mult_E(h) = mult_{E_{v_1}}(h) + mult_{E_{v_2}}(h) = 1+2=3$.  
\vskip0,2cm\noindent
If $(X,0)$ is a minimal singularity, we denote by $R(X)$ the minimal integer such that all test curves of $(X,0)$ have inner rates $<R(X)+1$.    We will achieve the proof that $(X,0)$ is LNE  by induction on $R(X)$. 

{We start with a lemma which will imply the first step of the induction and which will also be used in the induction step. 

\begin{lemma} \label{lem:first step} Let $(X,0)$ be a minimal singularity and   let $\ell (X,0) \to (\C^2,0)$ be a generic projection. Then every  test curve $\gamma$ for $\ell$ such that $q_{\gamma} \in \{1,3/2,4/3\}$ satisfies Condition  ($2^*$). 
\end{lemma}}

\begin{proof}
The case $q_{\gamma}=1$ has already been treated at the beginning of the proof. 

  Assume  that $q_{\gamma}=3/2$ and that we are in Case (1b), so $\gamma$ is a test curve at a $\Delta$-node $u$ of $T'$ such that $\widehat{s}(u)=3/2$. {By \eqref{cor:(ii)} of Corollary \ref{cor:graph map}), there is a   unique vertex $(v)$ in $G_0$ such that $L(v)=u$} and $(v)$ is  a $\cal P$-node obtained by blowing-up a central edge between two $\cal L$-nodes. Then $\ell^{-1}(\gamma)$ contains two principal components  $\widehat{\gamma}_1$ and $\widehat{\gamma}_2$ whose strict transforms by $\pi_0$ are curvettes of the curve $E_v$ meeting  $E_v$ at distinct points. We then have $q_{inn}(\widehat{\gamma}_1,\widehat{\gamma}_2)=3/2$. Let $\ell' \colon (X,0) \to (\C^2,0)$ be a generic projection for $(X,0)$ which is also generic for the curve $\widehat{\gamma}_1 \cup \widehat{\gamma}_2$. Then the restriction of $\ell'$ to $\widehat{\gamma}_1 \cup \widehat{\gamma}_2$ is a bilipschitz homeomorphism to its image, so we have $q_{out}(\widehat{\gamma}_1,\widehat{\gamma}_2) = q_{out}( \ell'(\widehat{\gamma}_1),\ell'(\widehat{\gamma}_2))$. But since $\ell'(\widehat{\gamma}_1)$ and $\ell'(\widehat{\gamma}_2)$ are two curvettes of the component $C'_u$ of $\rho_{\ell'}^{-1}(0)$ meeting $C'_u$  at distinct points, then $ q_{out}( \ell'(\widehat{\gamma}_1),\ell'(\widehat{\gamma}_2))$ equals the inner rate of $C'_u$, which is $3/2$. We then have obtain $q_{out}(\widehat{\gamma}_1,\widehat{\gamma}_2) = q_{inn}(\widehat{\gamma}_1,\widehat{\gamma}_2)$. 
  
 Assume  that $q_{\gamma}=3/2$ and that we are in Case (2a),  so $\gamma$ is a test curve at a vertex  $(u)$ of $T'$ obtained by blowing up the intersection point between the curve $C_{u'}$ of $\rho_{\ell}^{-1}(0)$ corresponding to the root-vertex $(u')$ of $T$ and a $\Delta$-curve $C_{u''}$ with inner rate $2$. 
    
  Let $\widehat{\gamma}_1, \widehat{\gamma}_2$ be a pair of principal components of $\ell^{-1}(\gamma)$ and let $(v_1)$ and $(v_2)$ be the two vertices of $\Gamma_0$ such that the strict transforms of  $\widehat{\gamma}_1, \widehat{\gamma}_2$ are curvettes   of  $E_{v_1}$ and  $E_{v_2}$ respectively.   By Proposition \ref{innergeomertythm}, $(v_1)$ and $(v_2)$ are obtained by blowing up  intersection points  $E_{v'_1} \cap E_{v''_1}$ and  $E_{v'_2} \cap E_{v''_2}$ such that $(v'_1)$ and $(v'_2)$ are $\cal L$-nodes and $(v''_1)$ and $(v''_2)$ are vertices of $G_0$ such that $L(v''_1) = L(v''_2)=(u'')$. Moreover, we have $(v'_1) \neq (v'_2)$.

  {${\mu_0\colon W_0}\to X$ factors through $NX$, so we have a composition ${W_0}\overset{{\tilde\mu_0}}\rightarrow NX\overset{\mathscr N}\rightarrow X$. Since  $E_{v_1}$ and $E_{v_2}$ are not components of the exceptional divisor of $\mathscr N$, the lifted Gauss map $\widetilde{\lambda} \colon NX \to \grassman(2,{\C^n})$  {(see Definition \ref{def:lifted Gauss map})} is constant on each subset ${\tilde\mu_0}(E_{v_1})$ and ${\tilde\mu_0}(E_{v_2})$ of $NX$.
 By  \cite[Proposition 7.2]{NPP}, in order to prove   Condition
 ($2^*$), i.e.,  $q_{inn}(\widehat{\gamma_1}, \widehat{\gamma_2}) =
 q_{out}(\widehat{\gamma_1}, \widehat{\gamma_2})$, it suffices to
 prove that  ${\widetilde\lambda}({\tilde\mu_0}(E_{v_1})) \neq
 {\widetilde\lambda}({\tilde\mu_0}(E_{v_2}))$,   i.e., that one
 has distinct values at the two points $p_1 =
 {E_{{v_1''}}} \cap E_{v'_1}
 { = \beta_0(E_{v_1})}$ and $p_2 = {E_{{v_2''}}}
 \cap E_{v'_2} {= \beta_0(E_{v_1})}$ of $\pi_0^{-1}(0)$. 
}

Let $\widetilde{e} \colon Y \to NX$ be the blow-up of the ideal $\cal M \cal O_{NX}$ and let $\mathscr N' \colon Y \to \tilde{X}$ be the morphism induced by the universal property of the blow-up, so we have a commutative diagram:
$$\xymatrix{
X \times  \grassman(2,\C^n) \times \mathbb P^{n-1}\supset 
 &Y \ar[r]^{\widetilde{e}}\ar[d]_{\mathscr N'} & NX  \ar[d]^{\mathscr N} & \subset X \times  \grassman(2,\C^n) \\
X\times \mathbb P^{n-1}  \supset  & \tilde{X} \ar[r]^e  & X&
}
$$
  where $\widetilde{e}$ and $\mathscr N'$ are the restriction to $Y$ of the canonical projections of $X \times  \grassman(2,\C^n) \times \mathbb P^{n-1}$ to  $X \times  \grassman(2,\C^n)$
 and  $X\times \mathbb P^{n-1}$  respectively. 

The curve  $e^{-1}(0)$ is isomorphic to the projectivized tangent cone
$C_0(X) \subset \mathbb P^{n-1}$ of $(X,0)$.  By  \cite[Theorem
3.4.6]{K}, the  projectivized tangent cone of a minimal singularity
consists of rational components and has only minimal singularities. So
the two curves $\mathscr N'({E_{v'_1}})$ and $\mathscr N'({E_{v'_2}})$  are two irreducible  components of  the projectivized tangent cone
intersecting transversely at a point $p$. Let $P_1$ and $P_2$ be the tangent lines at $p$ respectively to   $\mathscr N'({E_{w_1}})$ and $\mathscr N'({E_{w_2}})$, so we have $P_1 \neq P_2$.  {Let us denote by $P1 \in \grassman(2,\C^n)$ and $P2 \in \grassman(2,\C^n)$ the unprojectivized versions of  $P_{1}$ and $P_{2}$.}     Let $l_1,\ldots,l_r$ be the points of $C_0(X)$ which correspond to the exceptional tangent of $(X,0)$, so $p$ is one of them  {(see \cite{LT})}. 
Let  $(x_n)_n$ be a sequence of points on ${E_{w_1}}$ converging to $p_1$ such that none of $\mathscr N'(x_n)$ is an exceptional line. By \cite[Theorem 2.3.7]{L00},   $\widetilde{\lambda}(x_n)$ is the element of $\grassman(2,{\C^n})$ defined by the tangent line to $C_0(X)$ at $\mathscr N'(x_n)$. Taking the limit $x_n \to p_1$, we then obtain by continuity of $\widetilde{\lambda}$ that  $\widetilde{\lambda}(p_1) = P_{1}$. Similarly, we have  $\widetilde{\lambda}(p_2) = P_{2}$.  This proves $\widetilde{\lambda}(p_1) \neq \widetilde{\lambda}(p_2)$ which completes the proof of Condition ($2^*$) for $q_{\gamma}=3/2$ in case (2a).

 Finally assume $q_{\gamma}=4/3$.  Then $\gamma$ is a curvette of a component $C_u$ obtained by blowing up the intersection point between the root-curve $C_{u'}$ of $\rho_{\ell}^{-1}(0)$  and  a $\Delta$-curve  $C_{u''}$ with inner rate $3/2$.  {Using again \eqref{cor:(ii)} of Corollary \ref{cor:graph map}), there is a unique  vertex $(v'')$ in $G_0$ such that $L(v'')=u''$ and it is a $\cal P$-node $(v'')$}  obtained by blowing-up    a central edge  $(v'_1)-(v'_2)$  between  two $\cal L$-nodes in $G_0$, and $L^{-1}(u)$ consists in the two vertices $(v_1)$ and $(v_2)$  in $\Gamma_0$ obtained by blowing-up the two edges  $(v'_1)-(v)-(v'_2)$. Then we obtain the following string  in $\Gamma_0$:  $$(v'_1)-(v_1)-(v'')-(v_2)-(v'_2),$$
 and $\ell^{-1}(\gamma)$ contains exactly two principal components $\widehat{\gamma}_1, \widehat{\gamma}_2$ whose strict transforms are curvettes of $E_{v_1}$ and $E_{v_2}$ respectively.  Set $p_1 = E_{v'_1} \cap E_{v_1}$ and $p_2 = E_{v'_2} \cap E_{v_2}$. By the same argument as in the previous case, we have $\widetilde{\lambda}(p_1) \neq \widetilde{\lambda}(p_2)$, and this implies that the  pair  $\widehat{\gamma}_1, \widehat{\gamma}_2$ satisfies Condition ($2^*$). \end{proof}

\vskip0,2cm\noindent
{\bf First step of the induction.}
Assume that $R(X)  =1$.   We already know that every  test curve satisfies Condition ($1^*$).  Moreover,  the only possible values for inner rates of   test curves are $q_{\gamma} \in \{1,3/2,4/3\}$. Then by Lemma \ref{lem:first step}, every  test curve $\gamma$ satisfies  Condition  ($2^*$).

 \vskip0,2cm\noindent
{\bf Induction step.}   
We assume that all minimal singularities $(X,0)$ with $R(X) \leq n-1$ are LNE. Let $(X,0)$ be a minimal singularity with $R(X) = n$. {Consider a generic projection $\ell \colon (X,0) \to (\C,0)$ and a   test curve $(\gamma,0)$ for $\ell$ which is not a curvette of $e^{-1}(0)$.  We already know that every  test curve satisfies Condition ($1^*$). We have to prove that $(\gamma,0)$ satisfies  Condition  ($2^*$).

If $q_{\gamma} <2$, i.e., $q_{\gamma} \in \{1,3/2,4/3\}$, then  $\gamma$ satisfies  Condition  ($2^*$) by Lemma   \ref{lem:first step}.

Assume now that $q_{\gamma} \geq 2$ and let  $\widehat{\gamma}_1, \widehat{\gamma}_2$ be two principal components of $\ell^{-1}(\gamma)$. Consider the blow-up $e' \colon X' \to X$ and let $\widehat{\gamma}_1^*, \widehat{\gamma}_2^*$ be the strict transform by $e'$. As a direct consequence of \eqref{cor:(ii)} of Corollary \ref{cor:graph map}, $\widehat{\gamma}_1^*, \widehat{\gamma}_2^*$ pass through the same point $p$ of $X'$ which is a singular point of $X'$ having minimal resolution graph one of the components $G_p$ of $G$ minus its $\cal L$-nodes. Moreover, by Proposition \ref{prop:formula}, $(X',p)$ is a minimal singularity. }

  Denote by $s_p\colon V(G_p) \to \N$   the corresponding $s$-map
  defined in  Section \ref{sec:minimal singularities}. It follows
  immediately from Theorem \ref{thm:Pi} that  the  $\cal L$-nodes  of
  $G_p$ are the vertices of $G_p$ adjacent to the $\cal L$-nodes of
  $G_0$. As a consequence of this,  we have for  each vertex
  $(v) \in V(G_p)$, $s_p(v)=s(v)-1$, where $s$ is the $s$-map
  of the initial minimal singularity $(X,0)$. Therefore $R(X') = R(X)
  - 1$, so we can apply the induction hypothesis to $(X',p)$. Since
  $(X',p)$ is LNE we have $q'_{inn}(\widehat{\gamma}_1^*,
  \widehat{\gamma}_2^*)  = {q'_{out}}(\widehat{\gamma}_1^*,
  \widehat{\gamma}_2^*)$, {where $q'_{inn}$ and $q'_{out}$ are the
    inner and outer contacts in the germ $(X',p)$.}
  
  Let $\Gamma_0$ be  again the dual graph of the LNE$_{test}$-resolution  of $(X,0)$ (Definition \ref{def:LNE test resolution}). Then $V(G_0) \subset V(\Gamma_0)$  and by Proposition \ref{innergeomertythm}, the dual graph $\Gamma_p$ of the  LNE$_{test}$-resolution  of $(X',p)$ is the connected component of $\Gamma_0$ minus its $\cal L$-nodes which contains the vertices of $G_p$. If $(w)$ is a vertex of  $\Gamma_0$ (resp. $\Gamma_p$), let us denote by $q_w$ (resp. $q'_w$) the inner rate of $(w)$ in $\Gamma_0$, i.e., with respect to the Lipschitz geometry of $(X,0)$ (resp. in $\Gamma_p$, i.e., with respect to the Lipschitz geometry of $(X',p)$). Since   $s(v)=s_p(v)+1$ for  each vertex $(v)$ of $  G_p$, then  again by Proposition  \ref{innergeomertythm},  we have $q_v=q'_v+1$. Then applying Proposition \ref{lem:inner contact computation}, we obtain:  $q_{inn}(\widehat{\gamma}_1, \widehat{\gamma}_2) = q'_{inn}(\widehat{\gamma}_1^*, \widehat{\gamma}_2^*) +1$. 
  
On  the other hand, by Proposition \ref{prop:formula}, we have $q_{out}(\widehat{\gamma}_1, \widehat{\gamma}_2)  = {q'_{out}}(\widehat{\gamma}_1^*, \widehat{\gamma}_2^*) +1$.  We then obtain  $q_{inn}(\widehat{\gamma}_1, \widehat{\gamma}_2) = q_{out}(\widehat{\gamma}_1, \widehat{\gamma}_2)$. Therefore, $\gamma$ satisfies Condition ($2^*$). 
   
 Summarizing, we have proved that all test curves for $(X,0)$ satisfy Conditions ($1^*$) and ($2^*$). Therefore $(X,0)$ is LNE. This completes the induction step. 
\end{proof}

\section{Rational and LNE implies minimal}\label{sec:rational + NE implies minimal}

In this section, we prove the other direction of Theorem \ref{theorem:main}: any   rational surface singularity which is  LNE is minimal
\begin{remark}  An LNE surface singularity is
  not necessarily minimal.  A counter-example is given by the   (non
  rational) hyper surface in $\mathbb C^3$ with equation  $xy(x+y) +
  z^4=0$.  It is a super isolated singularity. The graph of its minimal
  resolution factorizing through Nash modification has four vertices. It consists
  of a central vertex and three bamboos of length one,  these three
  leaves being the $\cal L$-nodes, and the central vertex the single
  $\cal P$-node.
\end{remark}

\begin{proof}
Let $(\widetilde{X},E)$ be the minimal resolution of $(X,0)$, $Z$ the minimal
cycle and $E=\bigcup E_i$. {Since $(X,0)$ is LNE, then for every generic projection $\ell \colon (X,0) \to (\C^2,0)$, every generic line $(\gamma,0) \subset (\C^2,0)$ satisfies Condition ($1^*$) of Theorem \ref{cor:complex characterization of normal embedding}, i.e., every component of $\ell^{-1}(\gamma)$ has multiplicity $1$. This means that  the multiplicity of $Z$ at every $\cal L$-curve equals $1$. } 
 Consider Laufer's algorithm for
finding $Z$ (\cite[Proposition 4.1]{La}), and let $E_j\subset E$ be the last curve one adds
in the algorithm before one obtains $Z$. Assume that $E_j$ is not an
$\cal L$-curve, so $Z\cdot E_j=0$. Let $Z'$ be the penultimate cycle obtained by Laufer's algorithm. Then $Z'=Z-E_j$ and
$Z'\cdot E_j= -E_j^2>1$ which contradicts $(X,0)$ being rational by
Laufer's criterion \cite[Theorem 4.2]{La}. So the last curve added by Laufer's algorithm is
always an $\cal L$-curve. 

One can always run Laufer's algorithm such that each curve is
added once, before any curve is added a second time. So unless $Z=\sum
E_i$ there
would be an  $\cal L$-curve with multiplicity $>1$, which is a contradiction. Thus $(X,0)$ is minimal.
\end{proof}

\affiliationone{Walter D Neumann\\
Department of Mathematics\\
Barnard College, Columbia University\\
2009 Broadway MC4429\\
New York, NY  10027\\ USA \email{neumann@math.columbia.edu}
} 
\affiliationtwo{Helge M\o ller Pedersen\\
Departamento de Matem\'atica\\ Universidade Federal do Cear\'a\\
Campus do Pici, Bloco 914 \\ CEP 60455-760 \\ Fortaleza, CE \\Brazil \email{helge@mat.ufc.br}}

\affiliationthree{Anne Pichon\\
Aix Marseille Universit\'e, CNRS\\ 
Centrale Marseille, I2M, UMR 7373\\ 
13453 Marseille\\ France \email{anne.pichon@univ-amu.fr}}

\end{document}